\documentclass{amsart}

\usepackage{graphicx}
\usepackage{mathtools}

\newcommand{\typepm}{\mathring{\pm}}
\newcommand{\typemp}{\mathring{\mp}}

\newtheorem{theorem}{Theorem}[section]
\newtheorem{lemma}[theorem]{Lemma}

\theoremstyle{definition}
\newtheorem{definition}[theorem]{Definition}
\newtheorem{example}[theorem]{Example}

\newtheorem{remark}[theorem]{Remark} 
\theoremstyle{corollary}
\newtheorem{corollary}[theorem]{Corollary}

\theoremstyle{proposition}

\theoremstyle{conjecture}

\theoremstyle{conditionalproposition}

\numberwithin{equation}{section}

\begin{document}

\title[Bifurcations in Kac--Murdock--Szeg\H{o} matrices]{Eigenvalue bifurcations in Kac--Murdock--Szeg\H{o} matrices with a complex parameter}

\author{George Fikioris}

\address{School of Electrical and Computer engineering, National Technical University of Athens, GR 157-73 Zografou, Athens, Greece}
\email{gfiki@ece.ntua.gr}


\date{\today}


\begin{abstract}
For complex $\rho$, the spectral properties of the Toeplitz matrix  $K_{n}(\rho)=\left[\rho^{|j-k|}\right]_{j,k=1}^{n}$, often called the Kac–-Murdock--Szeg\H{o} matrix, have been examined in detail in two recent papers. The second paper, in particular, introduced the concept of borderline curves. These are two closed curves in the complex-$\rho$ plane that consist of all the $\rho$ for which $K_n(\rho)$ possesses some eigenvalue whose magnitude equals the matrix dimension $n$. The purpose of the present paper is to  examine eigenvalue bifurcations in both a qualitative and a quantitative manner, and to discuss connections between bifurcations and the borderline curves.
\end{abstract}

\maketitle

\section{Introduction}

This paper is a direct continuation of \cite{Fik2018} and \cite{FikMavr}. For $\rho \in \mathbb{C}$ and $n\ge 3$, we deal with the Toeplitz matrix 
\begin{equation}\label{matrixdefinition}
K_{n}(\rho)=\left[\rho^{|j-k|}\right]_{j,k=1}^{n}=\begin{bmatrix}
1 & \rho & \rho^2 & \ldots & \rho^{n-1} \\
\rho & 1 & \rho & \ldots & \rho^{n-2} \\
\rho^2 & \rho & 1 & \ldots & \rho^{n-3} \\
\vdots & \vdots & \vdots & \ddots & \vdots \\
\rho^{n-1} & \rho^{n-2} & \rho^{n-3} & \ldots & 1
\end{bmatrix}
\end{equation}
which, in the special case $0<\rho<1$, is usually called Kac--Murdock--Szeg\H{o} matrix; see \cite{Fik2018} for applications and a history.
As discussed in \cite{Fik2018}, studies of $K_n(\rho)$ have served as first steps for the development of general theorems pertaining to Toeplitz matrices, or as means of illustrating such theorems. This has been true ever since the original work by Kac, Murdock, and Szeg\H{o} \cite{KMS, Grenander} and~ remains~true~today~\cite{Bogoya}. 

Our overall purpose is to discuss eigenvalue bifurcations. We discover eigenvalue behaviors whose general characteristics are familiar, as they are ubiquitous in physical applications. For the case of $K_n(\rho)$, our specific goal is to study these behaviors, as well as related concepts, in detail. To do this, we exploit the implicit equations (developed in  \cite{Fik2018}, and containing an additional parameter $\mu\in\mathbb{C}$) that connect the eigenvalues $\lambda$ to the parameter $\rho$.

Let us give some indicative examples of our findings and explain their relevance. We track eigenvalue bifurcations as we move along selected paths in the complex-$\rho$ plane, which pass  through \textsl{critical} or \textsl{exceptional} points $\rho_c$ (we usually prefer the first term). Bifurcations of this sort are significant in problems of, among others, non-Hermitian quantum mechanics \cite{Heiss1}--\cite{Kanki} and optics, photonics, and electromagnetics \cite{Heiss2},  \cite{Figotin}--\cite{Zhiyenbayev}. Exceptional points, which give rise to many applications, occur in nonconservative or open physical systems which can interact with their environments. Moving along a path in the complex-$\rho$ plane corresponds to a certain eigenvalue trajectory in the complex-$\lambda$ plane; for the case where $\rho$ is close to $\rho_c$, we derive an explicit equation for this trajectory, which corresponds to the strong coupling discussed in \cite{Kirilov} and elsewhere. In the physical problems, it is often important whether \textit{all} matrix eigenvalues are real or not; see, e.g., the discussions on broken and unbroken $\mathcal{PT}$-symmetry in \cite{MiriAlu, BenderMakingSense}. While the spectrum of $K_n(\rho)$ is generally not exclusively real, we find that the bifurcating eigenvalues become real or ``nearly real'' \textit{after} bifurcation.   We find explicit expressions for the eigenvectors that coalesce/split at the critical points, and point out that they are isotropic vectors \cite{HornJohnson, Garcia} (often termed self-orthogonal vectors or isotropic null vectors in the physics literature \cite{Moiseyev, Scott}).

The aforementioned eigenvectors are self-orthogonal simply because the matrix $K_n(\rho)$ is complex symmetric \cite{Garcia}.  $K_n(\rho)$ also belongs to other matrix classes frequently investigated today. First, $K_n(\rho)$ is non-Hermitian (unless $\rho\in\mathbb{R}$, a case not of interest to us here). Though non-Hermitian Hamiltonians had been used as early as 1959 \cite{BenderMakingSense, Wu}, the study of non-Hermitian matrices (or operators) accelerated in 1998,  when Bender and Boettcher published their work \cite{BenderBoettcher}; it is indicative that most of the papers \cite{Heiss1}--\cite{BenderMakingSense} include some discussion of non-Hermitian matrices, even if  these are small ($2\times 2$ or $3\times 3$).
Second,  $K_n(\rho)$ is nonnormal (unless $\rho\in\mathbb{R}$ or $n=2$ \cite{Fik2018}, cases not of interest herein); nonnormal matrices and operators are analyzed in, e.g., the comprehensive 2005 book by Trefethen and Embree \cite{Trefethen},  
which is is intimately related to eigenvalue perturbations. Third, all bifurcations we find occur when  $|\rho|>1$, which is precisely when the associated Laurent matrix (i.e., the associated doubly infinite Toeplitz matrix) does not have a well-defined and bounded symbol. As discussed in \cite{Fik2018}, this hinders studies of spectral behavior of $K_n(\rho)$.  Investigations of a class of Toeplitz matrices (albeit Hermitian ones) whose associated Laurent matrices have no symbols can be found in the 2012 work by Bogoya, B\"ottcher, and Grudsky \cite{Bottcher-increasing}. The matrix $K_n(\rho)$, in other words, belongs to several interesting matrix classes.

The studies herein are intervolved with a number of results from the recent work \cite{FikMavr}. As a first example, eigenvalue bifurcations can be seen in Fig.~4 of \cite{FikMavr}, which was generated by computing the eigenvalues numerically. In the present paper, we focus on \textit{a priori} formulas, suitable for predicting bifurcation behaviors without actually finding eigenvalues. As an another example, \cite{FikMavr} introduces the so-called \textsl{type-}$k$ \textsl{borderline curve}, where $k=1,2$. This is a closed curve in the complex-$\rho$ plane that consists of the $\rho$ for which the magnitude of some type-1 or type-2 eigenvalue of $K_n(\rho)$ equals the matrix dimension $n$. In the present paper, we study what we call the \textsl{local level curve}. As we will see, each local level curve approaches a borderline curve; and both these curves are closely connected to eigenvalue bifurcations.

What are type-1 and type-2 eigenvalues? Each eigenvalue of $K_n(\rho)$ is of one of the two types. The most notable distinguishing feature is that type-1 (type-2) eigenvalues are associated with skew-symmetric (symmetric) eigenvectors. More precise information on the two types can be found in Theorem 3.7 of \cite{Fik2018}, Theorem 4.1 of \cite{Fik2018}, and Remark 4.2 of \cite{Fik2018}.

We proceed by discussing the matrix $K_3(\rho)$, i.e., the special case $n=3$, which serves as a motivating introduction to---and an illustrative example for---many of the concepts developed in this paper.

\subsection{Introductory example: The case $n=3$}
\label{section:n-equals-3}

This section examines $K_3(\rho)$ using elementary methods.\footnote{The case $n=4$ can also be studied by elementary methods, because all four eigenvalues are given by explicit formulas that are similar to (\ref{eq:exacteigenvaluefor3}). When $n=5$, like formulas also give the two type-1 eigenvalues that coalesce/split when $\rho=\rho_c=2i$. For $n\ge 3$, these seem to be the only cases that are amenable to analysis by elementary methods of this sort.} For brevity, we also use a result from \cite{FikMavr} relevant to level curves.

For any $\rho\in \mathbb{C}\setminus\{-1,0,1\}$, it is easily found that $K_3(\rho)$ has the type-1 eigenvalue $1-\rho^2$ and that the other two eigenvalues are of type-2 and equal to
\begin{equation}
\label{eq:exacteigenvaluefor3}
\lambda(\rho)=\frac{1}{2}\left(2+\rho^2\pm\rho\sqrt{\rho^2+8}\right)
\end{equation}
where, consistently with what is to follow, we use the symbol $\lambda(\rho)$ for \textit{two} eigenvalues. We see that  repeated eigenvalues occur (apart from the excluded cases $\rho=0$ and $\rho=\pm 1$) iff $\rho$ assumes the two critical values $\rho=\rho_c=\pm i\sqrt{8}$, which are branch points of the double-valued function $\lambda(\rho)$ given by (\ref{eq:exacteigenvaluefor3}). Both repeated eigenvalues are (algebraically) double eigenvalues equal to $-3$. Thus the matrix $K_3\left(i\sqrt{8}\right)$ possesses the type-2 double eigenvalue
$\lambda_c=-3$; it is then easy to verify that the geometric multiplicity of this double eigenvalue is 1, and that it is associated with the eigenvector $\textit{\textbf{v}}_c$, where $\textit{\textbf{v}}_c^T=(1,-i\sqrt{2},1)$. Similarly, $K_3\left(-i\sqrt{8}\right)$ possesses the type-2 double eigenvalue
$\lambda_c=-3$, having geometric multiplicity 1 and associated with the eigenvector $\textit{\textbf{v}}_c$, where $\textit{\textbf{v}}_c^T=(1,i\sqrt{2},1)$.

With $\rho=iy$, (\ref{eq:exacteigenvaluefor3}) becomes
\begin{equation}
\label{eq:exacteigenvaluefor3-in-terms-of-y}
\lambda(\rho)=\lambda(iy)=\frac{1}{2}\left[2-y^2\pm y\sqrt{\left(y-\sqrt{8}\right)\left(y+\sqrt{8}\right)}\right].
\end{equation}
Let us track the behavior of our two type-2 eigenvalues as we start at the origin and move upwards along the imaginary axis, so that $y>0$. Initially, the two eigenvalues have nonzero imaginary parts and are complex conjugates. After passing the point $\rho_c=i\sqrt{8}$, however, the two split (from their common value $\lambda_c=-3$) and remain real. In fact, for $d=y-\sqrt{8}>0$, the right-hand side of  (\ref{eq:exacteigenvaluefor3-in-terms-of-y}) can be expanded into a convergent power series in the variable $\sqrt{d}$. Specifically,
\begin{equation}
\label{eq:puiseuxfor3}
\frac{\lambda(\rho)}{\lambda_c}=\frac{\lambda(\rho)}{-3}=1\mp\frac{2^{1/4}\sqrt{8}}{3}\sqrt{d}+\frac{\sqrt{8}}{3} d+O\left(d^{3/2}\right),\quad \mathrm{as}\quad d\to 0+0.
\end{equation}
Thus, upon passing $\rho_c$, the eigenvalues exhibit a square-root behavior. Similar behavior occurs if we move downwards from the origin and pass the other critical point ($\rho_c=-i\sqrt{8}$).
\begin{figure}
\begin{center}
\includegraphics[scale=0.4]{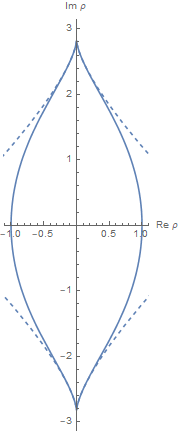}
\end{center}
\caption{Introductory case $n=3$: Type-2 borderline curve $B_3^{(2)}$ (solid line) together with the two local level curves (two sets of dashed lines) associated with the critical points $\rho_c=\pm i\sqrt{8}$.} \label{fig:Fig1}
\end{figure}

Let us return to complex $\rho$.
We have $|\lambda\left(i\sqrt{8}\right)|=|\lambda\left(-i\sqrt{8}\right)|=3$. Therefore, in the complex $\rho$-plane, the level curve (contour line) defined by the equation $|\lambda(\rho)|=3$ passes through the two critical points $\rho_c$. An implicit equation for this level curve can be easily found from (\ref{eq:exacteigenvaluefor3}). The level curve can also be determined by the methods of \cite{FikMavr}, where it is called the \textsl{type-2 borderline curve} and is denoted by $B_n^{(2)}=B_3^{(2)}$. The curve $B_3^{(2)}$, shown as the solid line in Fig. \ref{fig:Fig1}, is symmetric with respect to the imaginary axis and is seen to have cusp-like singularities at the two critical points.\footnote{The type-1 eigenvalue $1-\rho^2$ is similarly associated with the level curve $|1-\rho^2|=3$, denoted in \cite{FikMavr} by $B_3^{(1)}$. This curve is a single-loop Cassini oval. But for all $n\ge 3$, it is unique among the borderline curves $B_n^{(k)}$ because it has no cusp-like singularities \cite{FikMavr}. It is not of interest to us here.} Therefore the imaginary axis---which we related to the above-described bifurcations---can be viewed  as the unique straight line that: (i) bisects the two cusp-like singularities;  and (ii) is tangent to the two branches of $B_3^{(2)}$ (see also \cite{FikMavr}). Consequently, when we wish to observe eigenvalue bifurcations by moving along a straight line, the imaginary  axis is the natural line to choose.

In (\ref{eq:puiseuxfor3}) we have written the first three terms of the \textsl{Puiseux series} of the double-valued function $\lambda(\rho)$. Such a series makes sense (as a convergent series, with an infinite number of terms) for any $\rho\in\mathbb{C}$ sufficiently close to $\rho_c$, and consists of integer powers of $\epsilon^{1/2}$, where $\epsilon=\rho-\rho_c=\rho-i\sqrt{8}$. 

\subsection{Purpose of this paper} 
\label{section:purpose}
In this paper, we build upon the Puiseux series and on certain results from \cite{Fik2018} and \cite{FikMavr} (the said results are reviewed in Section \ref{section:preliminaries}) in order to discuss bifurcations and related concepts associated with $K_n(\rho)$, for all $n\ge 3$. We manage to explicitly determine only the first three terms of the Puiseux series, something that limits our formulas' numerical accuracies. Retaining three terms, however, does allow us to explicitly compute our ``local level curves'' and to draw a number of conclusions about eigenvalue behavior

As we will see in Section \ref{section:preliminaries}, all nontrivial repeated eigenvalues of $K_n(\rho)$ are double eigenvalues equal to $-n$. Initially (Sections \ref{section:analytic-matrix-functions}--\ref{section:puiseux-for-KMS}) we consider \textit{any} $\rho_c\in\mathbb{C}$ for which such a double eigenvalues occurs. We then (Section \ref{section:purely-imaginary-rho}) focus on critical points $\rho_c=iy_{n}$ that lie on the imaginary-$\rho$ axis; in this case, most results simplify considerably, especially (Section \ref{section:large-n}) when the matrix dimension $n$ becomes large. (Except for Section \ref{section:large-n}, in this paper $n$ is arbitrary but fixed.)

As a rule, double eigenvalues of a square matrix $A_n(\rho)$ are associated with Puiseux series consisting of powers of $\epsilon^{1/2}=(\rho-\rho_c)^{1/2}$, provided that $A_n(\rho)$ is an analytic matrix function, meaning that the matrix entries are analytic functions of $\rho\in\mathbb{C}$ \cite{Lancaster}--\cite{Kato}. For this reason, we have formulated some of our results (Sections \ref{section:analytic-matrix-functions} and \ref{section:puiseux-for-analytic}) so as to be more generally applicable. Section \ref{section:puiseux-for-analytic}, especially, contains formulas for the first three terms of the Puiseux series of an analytic matrix function $A_n(\rho)$ for which there exists an implicit relation---via analytic functions $\lambda(\mu)$ and $\rho(\mu)$ of a new parameter $\mu$---between $\rho$ and the eigenvalues $\lambda$. This assumption renders the formulas in Section \ref{section:puiseux-for-analytic}---which are limited to the case of a double eigenvalue---different from a number of formulas in the literature giving the first few terms of Puiseux series. For such formulas, whose starting point is often the Jordan structure of the matrix, see \cite{Kirilov, Lancaster}, \cite{Vishik}--\cite{Moro}, the discussions in pp. 477--478 of \cite{Trefethen}, and the comprehensive analysis and bibliography in \cite{Welters}. 

As seen in Fig. \ref{fig:Fig1}, $B_3^{(2)}$ occupies a finite region in the complex-$\rho$ plane, something that remains true for any $B_n^{(k)}$ \cite{FikMavr}. A borderline curve must thus be regarded as a global property of $K_n(\rho)$. Throughout, we  are concerned with \textit{local} rather than global properties.  
In particular, we revise and generalize the ``local level curve'' that was used (specifically for the case of $K_n(\rho)$) in \cite{FikMavr}. For the case $n=3$, the two local level curves (one associated with $\rho_c=+i\sqrt{8}$, the other with $\rho_c=-i\sqrt{8}$, and both computed by methods to be developed in the present paper) are shown as the dashed lines in  Fig. \ref{fig:Fig1}.

\subsection{Notes on notation}
\label{section:notation}

Throughout, ``analytic'' function always means complex-analytic function. The two symbols $\sqrt{z}$ and $z^{1/2}$ can denote either the single- or the double-valued square root of $z\in\mathbb{C}$; when not clarified, the meaning should be clear from the context. When we wish to emphasize that two values are meant, we use a plus/minus sign and write $\pm \sqrt{z}$ or $\pm z^{1/2}$ (we already did this in Section \ref{section:n-equals-3}).

Most equations for type-1 and type-2 eigenvalues differ only by sign changes. Thus, equations pertaining to both types can be compactly written with the aid of a plus/minus sign. To distinguish from the  aforementioned (ordinary) symbols $\pm$ and $\mp$, we use the special symbols $\typepm$ and $\typemp$. Eqn. (\ref{eq:rho-prime-in-terms-of-mu}), for example, should be interpreted as the following two,
\begin{equation*}
\rho^{\prime}(\mu)=\frac{-(\lambda(\mu)+n)\sin\mu}{1- \cos[(n-1)\mu]},\qquad
\rho^{\prime}(\mu)=\frac{-(\lambda(\mu)+n)\sin\mu}{1+\cos[(n-1)\mu]},
\end{equation*}
where the first (second) equation holds for type-1 (type-2) eigenvalues.

For $k=0,1,2,\ldots$, $T_k(t)$ and $U_k(t)$ denote the usual Chebyshev polynomials of the first and second kinds, given by
\begin{equation}
\label{eq:chebyshev}
T_k(t)=\cos\left(k\mathrm{Arccos}\,t\right),\quad U_k(t)=\frac{\sin[(k+1)\mathrm{Arccos} t]}{\sin(\mathrm{Arccos} t)},\quad t\in\mathbb{C}.
\end{equation}
Appendix \ref{app:b} lists some useful  to us properties of $T_k(t)$ and $U_k(t)$.

Theorem \ref{th:preliminary-t}  further uses the symbols $T_k(t)$ and $U_k(t)$ for $k=\frac{1}{2},\frac{3}{2},\frac{5}{2}\ldots$. In this case, $T_k(t)$ and $U_k(t)$ must be interpreted via (\ref{eq:chebyshev}). (It is easily verified that the multivaluedness of $\mathrm{Arccos}\,t$ does not affect Theorem \ref{th:preliminary-t}.)

Finally, RHS stands for right-hand side (of an equation).

\section{Results from \cite{Fik2018} and \cite{FikMavr}; corollaries}
\label{section:preliminaries}

This section gives our starting points, which are certain results from \cite{Fik2018} and \cite{FikMavr}, as well as some corollaries of those results.  Our first result is Lemma 2.2 of \cite{FikMavr} (for the case of odd $n$ only):

\begin{lemma}
\label{lemma:symmetries}
\cite{FikMavr}
Let $\rho\in\mathbb{C}$ and $n=3,5,7,\ldots$. Let $\lambda$ be a type-1 (or type-2) eigenvalue of $K_n(\rho)$. Then $\lambda$ is a type-1 (or type-2) eigenvalue of $K_n(-\rho)$. Furthermore, the complex conjugate $\overline{\lambda}$ is a type-1 (or type-2) eigenvalue of $K_n(\overline\rho)$.
\end{lemma}
Lemma \ref{lemma:symmetries} immediately gives the result that follows. It shows that, under certain conditions, the non-Hermitian, complex-symmetric matrix $K_n(\rho)$  possesses a property usually associated with matrices whose entries are real.
\begin{corollary}
\label{corollary:symmetries} Let $\rho=iy$ where 
$y\in\mathbb{R}$, and let $n=3,5,7,\ldots$. If $\lambda\in\mathbb{C\setminus R}$ is an eigenvalue of $K_n(iy)$, then $\overline{\lambda}$ is also an eigenvalue of  $K_n(iy)$. Furthermore, $\lambda$ and $\overline{\lambda}$ are eigenvalues of the same type (1 or 2).
\end{corollary}

The theorem that follows comes from Lemma 2.3 of \cite{FikMavr}. It gives the aforementioned implicit equations relating $\lambda$ and $\rho$.

\begin{theorem}\label{th:preliminary-mu} \cite{Fik2018}, \cite{FikMavr}
Let $\rho \in \mathbb{C}$. $\lambda \in \mathbb{C}$ is a type-1 (or type-2) eigenvalue of $K_n(\rho)$ iff
\begin{equation}\label{eq:lambda-of-mu}
\lambda=-\frac{\sin(n\mu)}{\sin{\mu}} \quad \left( {\rm or}\,\,\lambda=\frac{\sin(n\mu)}{\sin{\mu}}  \right),
\end{equation}
where $\mu\in \mathbb{C}$ satisfies
\begin{equation}\label{eq:rho-of-mu}
\rho=\frac{\sin{\frac{(n+1)\mu}{2}}}{\sin{\frac{(n-1)\mu}{2}}} \quad \left({\rm or}\,\,  \rho=\frac{\cos{\frac{(n+1)\mu}{2}}}{\cos{\frac{(n-1)\mu}{2}}}\right).
\end{equation}
\end{theorem}

Theorem \ref{th:preliminary-mu} naturally lends itself to the study of repeated eigenvalues:  Differentiating (\ref{eq:rho-of-mu}) with respect to $\mu$ and invoking (\ref{eq:lambda-of-mu}) we obtain
\begin{equation}
\label{eq:rho-prime-in-terms-of-mu}
\rho^{\prime}(\mu)=\frac{-(\lambda(\mu)+n)\sin\mu}{1\typemp \cos[(n-1)\mu]}
\end{equation}
(see Section \ref{section:notation} for the meaning of $\typemp$). Apart from certain trivial cases, (\ref{eq:rho-prime-in-terms-of-mu}) shows that the conditions $\rho^\prime(\mu)=0$ and $\lambda(\mu)=-n$ are equivalent. Either condition signals the appearance of a repeated eigenvalue, which is always a double eigenvalue and is necessarily equal to $-n$.

Repeated eigenvalues are discussed more precisely in the theorem that follows, which is\footnote{Eqn. (\ref{eq:rho-of-t-type-1}) is exceptional because, at first sight, it appears different from the corresponding equation in \cite{Fik2018} (see (4.14) of \cite{Fik2018}). To show our (\ref{eq:rho-of-t-type-1}), set $\mu=\mathrm{Arccos}\, t_0$ in (4.20) of \cite{Fik2018}, solve for $\rho$, and use our (\ref{eq:chebyshev}).} Theorem 4.5 of \cite{Fik2018}. Our theorem is expressed in terms of a parameter $t$, which is related to the parameter $\mu$ via
\begin{equation}
t=\cos\mu.
\end{equation}

\begin{theorem}
\label{th:preliminary-t}
\cite{Fik2018}
Let $\rho\in\mathbb{C}\setminus\{-1,0,1\}$ and $n=3,4,5,\ldots$. If $\lambda=\lambda_c$ is a repeated eigenvalue of $K_n(\rho)$, then it is a double eigenvalue with
\begin{equation}
\label{eq:lambdac-equals-minus-n}
\lambda_c=-n.
\end{equation}
Moreover, $\lambda_c=-n$ is a double eigenvalue of type-1 iff $\rho$ is equal to a value $\rho_c$ given by
\begin{equation}
\label{eq:rho-of-t-type-1}
\rho_c=\frac{U_{(n-1)/2}(t_c)}{U_{(n-3)/2}(t_c)},
\end{equation}
where $t_c\in\mathbb{C}$ is any zero of the polynomial
\begin{equation*}
q_1(n,t)=
\begin{cases}
\frac{U_{n-1}(t)-n}{t-1},\quad n=4,6,8,\ldots,\\
\frac{U_{n-1}(t)-n}{t^2-1},\quad n=5,7,9,\ldots.
\end{cases}
\end{equation*}
Similarly, $\lambda_c=-n$ is a double eigenvalue of type-2 iff $\rho$ is equal to a value $\rho_c$ given by
\begin{equation}
\label{eq:rho-of-t-type-2}
\rho_c=\frac{T_{(n+1)/2}(t_c)}{T_{(n-1)/2}(t_c)}.
\end{equation}
In this case, $t_c\in\mathbb{C}$ is any zero of the polynomial
\begin{equation*}
q_2(n,t)=
\begin{cases}
\frac{U_{n-1}(t)+n}{t+1},\quad n=4,6,8,\ldots,\\
U_{n-1}(t)+n,\quad n=3,5,7,\ldots.
\end{cases}
\end{equation*}
\end{theorem}

We call the various $t_c$ and $\rho_c$ the \textsl{critical values} of the parameters $t$ and $\rho$. We use the same term for the $\mu_c$ defined via
\begin{equation}
\label{eq:tc-mc}
t_c=\cos\mu_c.
\end{equation}

The theorem that follows, which gives  $\lambda$-eigenvectors in terms of $\mu$, comes from Theorem 4.1 of \cite{Fik2018}.
\begin{theorem}
\label{th:eigenvectors}
\cite{Fik2018}
Let $\rho\in\mathbb{C}\setminus\{-\frac{n+1}{n-1},-1,1,\frac{n+1}{n-1}\}$ and 
let $\lambda$ be an eigenvalue of $K_n(\rho)$ found via Theorem \ref{th:preliminary-mu}, with $\mu\ne 0, \pm \pi, \pm 2\pi,\ldots$. If $\lambda$ is of type-1,  then it is associated with an eigenvector $\textbf{v}\in \mathbb{C}^n$ whose elements $v_j$ are 
\begin{equation*}
v_j=\sin\left[\mu\left(j-\frac{n-1}{2}\right)\right],\quad j=0,1,\ldots,n-1.
\end{equation*}
If $\lambda$ is of type-2, it is associated with an eigenvector $\textbf{v}\in \mathbb{C}^n$ whose elements $v_j$ are
\begin{equation*}
v_j=\cos\left[\mu\left(j-\frac{n-1}{2}\right)\right],\quad j=0,1,\ldots,n-1.
\end{equation*}
\end{theorem}
Theorem \ref{th:eigenvectors} holds for any $\lambda$ found via Theorem \ref{th:preliminary-mu} and, in particular, for any double eigenvalue $\lambda=\lambda_c=-n$.\footnote{By Proposition 6.1 of \cite{Fik2018}, the excluded values of $\rho$ and $\mu$ in Theorem \ref{th:eigenvectors} cannot lead to a double eigenvalue equal to $-n$; see also Section~3 of \cite{FikMavr}.} In that case, the $\textbf{\textit{v}}=\textbf{\textit{v}}_c$ of Theorem \ref{th:eigenvectors} is the coinciding limit (as $\rho\to\rho_c$) of two eigenvectors that correspond to two coalescing eigenvalues. We can say slightly more about this $\textbf{\textit{v}}_c$:

An \textsl{isotropic vector} $\textbf{\textit{u}}=(u_0,u_1,\ldots,u_{n-1})$ is one for which $\sum_{j=0}^{n-1} u_j^2=0$ \cite{HornJohnson, Garcia}. Zero is the only isotropic vector in $\mathbb{R}^n$, so nonzero isotropic vectors are complex-valued. Furthermore, it is known (Thoerem 4.12 of \cite{Garcia}) that any repeated eigenvalue of a complex-symmetric matrix is necessarily associated with an eigenvector that is isotropic. 
\begin{corollary}
\label{corollary:isotropic}
Let $\lambda=\lambda_c=-n$ be a double eigenvalue of $K_n(\rho_c)$, as described in Theorem \ref{th:preliminary-t}. Then the associated eigenvector $\textbf{v}=\textbf{\textit{v}}_c$ provided by Theorem \ref{th:eigenvectors} is isotropic.
\end{corollary}
\begin{proof}
Sum the squares of the $v_j$ given in Theorem \ref{th:eigenvectors} and then use (\ref{eq:lambda-of-mu}) to obtain
\begin{equation*}
\sum_{j=0}^{n-1} v_j^2=\frac{1}{2}\left[n\typemp \frac{\sin (n\mu)}{\sin \mu}\right]=\frac{1}{2}\left(n+\lambda\right).
\end{equation*}
As $\lambda=\lambda_c=-n$, we see that $\textit{\textbf{v}}_c$ is an isotropic vector.
\end{proof}

Evidently, many statements about $K_n(\rho)$ that are formulated in terms of $\mu$ (and/or $\mu_c$) can be alternatively formulated in terms of $t$ (and/or $t_c$). For the calculations performed in \cite{Fik2018}, \cite{FikMavr}, and  the present paper, each formulation has its own advantages and drawbacks. There seems to be no essential reason to use one formulation exclusively and we use both in the present paper. 

\section{Analytic matrix functions}
\label{section:analytic-matrix-functions}

For an analytic matrix function, this section deals with eigenvalue behavior near a double eigenvalue. Our discussions proceed from the Puiseux series, are suitable for specialization to the Kac--Murdock--Szeg\H{o} matrix $K_n(\rho)$, and are self-contained. For reasons already noted in Section \ref{section:purpose}, we truncate the Puiseux series to three terms. 

\subsection{Analytic matrix functions; Puiseux-series parameters}
\label{section:preamble}
We are concerned with an \textsl{analytic matrix function} $A_n(\rho)$ ($\rho\in\mathbb{C}$), by which we mean a complex-valued $n\times n$ matrix whose elements are analytic functions of $\rho$ in some region of a point $\rho_c$. We take $\rho_c\in\mathbb{C}$ to be such that $A_n(\rho_c)$ has an eigenvalue of algebraic multiplicity 2, which we denote by $\lambda_c$.  We further assume the existence of a (convergent) Puiseux series  that takes the usual form \cite{Lancaster}--\cite{Kato}
\begin{equation}
\label{eq:puiseux-original}
\lambda (\rho) =\lambda_c\pm \alpha\epsilon^{1/2}+\beta\epsilon+O\left(\epsilon^{3/2}\right),\quad \mathrm{as}\quad\epsilon\rightarrow 0,
\end{equation}
where 
\begin{equation}
\epsilon=\rho-\rho_c,
\end{equation}
and $\lambda_c,\,\alpha,\,\beta\in\mathbb{C}$.
The $\pm$ sign in (\ref{eq:puiseux}) corresponds to two eigenvalues $\lambda(\rho)$. We finally assume that $\lambda_c$ and $\alpha$ are nonzero and rewrite (\ref{eq:puiseux-original}) as
\begin{equation}
\label{eq:puiseux}
\lambda (\rho) =\lambda_c\left[1\pm a\epsilon^{1/2}+b \epsilon+O\left(\epsilon^{3/2}\right)\right],\quad \mathrm{as}\quad\epsilon\rightarrow 0,
\end{equation}
where 
\begin{equation}
\label{eq:gammanonzero}
\lambda_c\ne 0 \quad\mathrm{and}\quad a\neq 0.
\end{equation}
We call $a$ and $b$ the \textsl{Puiseux-series parameters.}

We denote
\begin{equation}
\label{eq:notation}
\epsilon=|\epsilon|\exp(i\theta),\quad
a=|a|\exp(i\theta_a),\quad
b=|b|\exp(i\theta_b),\quad \Theta=\theta_b-2\theta_a,
\end{equation}
so that
\begin{equation}
\label{eq:approxlambda}
\frac{\lambda (\rho)}{\lambda_c} = 1\pm|\epsilon|^{1/2}|a|\exp\left[i\left(\frac{\theta}{2}+\theta_a\right)\right]+|\epsilon||b|\exp\left[i\left(\theta+\theta_b\right)\right]+O\left(|\epsilon|^{3/2}\right).
\end{equation}
The particular choices of the multi-valued $\theta$, $\theta_a$, $\theta_b$ are irrelevant to our purposes. The double-valued ratio $\lambda(\rho)/\lambda_c$ will be referred to as \textsl{normalized~eigenvalues}.
\begin{remark} For the moment, we \textit{postulated} the existence of the Puiseux series. Section~ \ref{section:puiseux-for-analytic}, however, will develop sufficient conditions (on $A_n(\rho)$ and $\rho_c$) for the series to exist.
Let us stress that double eigenvalues are not necessarily associated with Puiseux series. A simple example is afforded by the $3\times 3$ matrix $K_3(\rho)$: When $\rho=1$, $K_3(1)$ has the double eigenvalue $\lambda_c=0$, resulting from the coalescence of the two eigenvalues $1-\rho^2$ and $\left(2+\rho^2-\rho\sqrt{\rho^2+8}\right)/2$ (see Section \ref{section:n-equals-3}). While both eigenvalues are analytic at $\rho=1$, the two analytic functions differ. Consequently, there is no Puiseux series like (\ref{eq:puiseux-original}). Such eigenvalues are not associated with bifurcations and will not concern us in this paper.
\end{remark}
\begin{remark}
The assumptions in (\ref{eq:gammanonzero}) imply that two eigenvalues, which differ when $\rho\ne\rho_c$, coincide when $\rho=\rho_c$. In other words, the two depart from the unperturbed double eigenvalue $\lambda_c$ by splitting at $\rho=\rho_c$. We have thus assumed that the critical point $\rho_c$ is an \textsl{exceptional point} in the sense of Kato, see p. 66 of \cite{Kato}. The same assumption further implies that $\rho_c$ is a branch point of $\lambda(\rho)$. (Kato remarks that an exceptional point need not be a branch point.)
\end{remark}

\subsection{Local level curve; cusp bisector}

In the complex-$\rho$ plane, the two-branched level curve (contour line) $|\lambda(\rho)/\lambda_c|=1$ passes through the point $\rho=\rho_c$. 
\textit{Near} $\rho=\rho_c$, we approximate the aforementioned level curve by another two-branched curve that we define as follows.

\begin{definition}
\label{definition:local-level-curve}
By \textsl{local level curve} we mean the locus of points $\rho\in\mathbb{C}$ such that
\begin{equation}
\label{eq:localleveldefinition}
\left|\frac{\lambda (\rho)}{\lambda_c}\right|=1+O\left(|\epsilon|^{3/2}\right),\quad \mathrm{as}\quad\epsilon\rightarrow 0.
\end{equation}
\end{definition}

  The theorem that follows is an explicit equation for the local level curve.

\begin{theorem}
\label{th:local-level-curve} 
Let $A_n(\rho)$ be an analytic matrix function that is subject to the restrictions and definitions of Section \ref{section:preamble}. Assume that the additional condition
\begin{equation}
\label{eq:condition}
|a|^2-2|b|\cos\Theta\ne 0
\end{equation}
is satisfied. Then the local level curve is given (in a neighborhood of $\rho_c$) by the points
\begin{equation}
\rho=\rho_c+|\epsilon|e^{i\theta},
\end{equation}
in which, for $\theta$ sufficiently close to $\pi-2\theta_a$,
\begin{equation}
\label{eq:near-cardioid}
|\epsilon|=|\epsilon|(\theta)=\left(\frac{2|a|\cos\left(\frac{\theta}{2}+\theta_a\right)}{|a|^2\sin^2(\frac{\theta}{2}+\theta_a)+2|b|\cos\left(\theta+\theta_b\right)} \right)^2.
\end{equation}
Thus
in polar coordinates $(r,\theta)=(|\epsilon|,\theta)$, the local level curve is given by the polar equation (\ref{eq:near-cardioid}).
\end{theorem}

\begin{proof}
Take the magnitude-squared of (\ref{eq:approxlambda}) to get
\begin{equation}
\label{eq:approxlambdamag1}
\left|\frac{\lambda (\rho)}{\lambda_c}\right|^2 = 1\pm 2|\epsilon|^{1/2}|a| \cos\left(\frac{\theta}{2}+\theta_a\right)+|\epsilon|\left[|a|^2+2|b|\cos\left(\theta+\theta_b\right)\right]+O\left(|\epsilon|^{3/2}\right).
\end{equation}
Taking the square root yields
\begin{multline}
\label{eq:approxlambdamag2}
\left|\frac{\lambda (\rho)}{\lambda_c}\right| = 1\pm |\epsilon|^{1/2}|a| \cos\left(\frac{\theta}{2}+\theta_a\right)+\frac{|\epsilon|}{2}\left[|a|^2\sin^2\left(\frac{\theta}{2}+\theta_a\right)+2|b|\cos\left(\theta+\theta_b\right)\right]\\+O\left(|\epsilon|^{3/2}\right).
\end{multline}
For $\theta=\pi-2\theta_a$, the  $O(|\epsilon|)$ term in (\ref{eq:approxlambdamag2}) is ensured by the condition (\ref{eq:condition}) to be nonzero. This term must remain nonzero for all $\theta$ sufficiently close to $\pi-2\theta_a$. For such $\theta$, therefore, the RHS of (\ref{eq:near-cardioid}) is well defined. It is also non-negative and it vanishes when $\theta=\pi-2\theta_a$. When $|\epsilon|$ is given by this RHS, and only then, the combination of the $O(|\epsilon|^{1/2})$ and $O(|\epsilon|)$ terms in (\ref{eq:approxlambdamag2}) vanishes, leaving $|\lambda(\rho)/\lambda_c|=1+O(|\epsilon|^{3/2})$, in accordance with Definition \ref{definition:local-level-curve}. 
Consequently, a point $\rho=\rho_c+|\epsilon|\exp(i\theta)$ belongs to the local level curve iff it lies on the continuous curve described by (\ref{eq:near-cardioid}).
\end{proof}

As $\theta\to \pi-2\theta_a$, it is easy to see that the local level curve in (\ref{eq:near-cardioid}) can further be approximated by another curve as follows.
\begin{equation}
\label{eq:cardioid}
|\epsilon|=|\epsilon|(\theta)\sim
\frac{4\sin^2\frac{\theta-(\pi-2\theta_a)}{2}}{\left(|a|^2-2|b|\cos\Theta\right)^2},\quad \theta\to \pi-2\theta_a.
\end{equation}
Unlike the local level curve, the above curve is symmetric about $\theta=\pi-2\theta_a$. Eqn. (\ref{eq:cardioid}), which is the polar equation of a cardioid \cite{Lawrence}, leads to the following result.
\begin{corollary}
\label{corollary:cusp}
The local level curve has a cusp at the point $\rho=\rho_c$ (or $|\epsilon|=0$).
\end{corollary}

Each value of $\theta=\arg(\epsilon)$ corresponds to a ray (half-line) that emanates from the point 
$\rho=\rho_c$  in the complex-$\rho$ plane. The ray $\theta=\pi-2\theta_a$ is exceptional in the following sense.

\begin{corollary}
\label{corollary:cusp-bisector}
The ray $\theta=\pi-2\theta_a$ bisects the cusp and  is tangent, at $\rho=\rho_c$, to both branches of the local level curve of Definition \ref{definition:local-level-curve}. We call this ray the \normalfont{cusp bisector}.
\end{corollary}

\begin{proof}[Proof of Corollaries \ref{corollary:cusp} and \ref{corollary:cusp-bisector}]
Since (\ref{eq:cardioid}) is the polar equation of a cardioid, the desired local properties of the local level curve in (\ref{eq:near-cardioid}) follow \cite{Courant} by comparing to familiar local properties of the cardioid: The cardioid has a cusp at the origin; the ray $\theta=\pi-2\theta_a$ bisects the cusp; and this ray is tangent to the two branches of the cardioid \cite{Lawrence, Courant}. Note that $\theta=\pi-2\theta_a$ is the only ray that has one and only one point in common with the cardioid, the point being the origin $|\epsilon|=0$.
\end{proof}
\begin{figure}
\begin{center}
\includegraphics[scale=0.4]{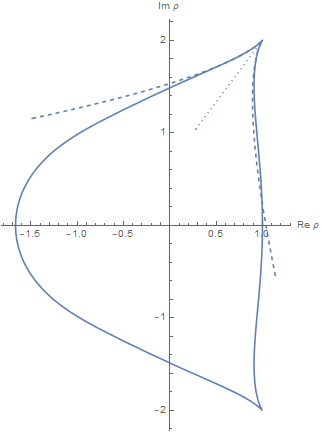}
\end{center}
\caption{Type-2 borderline curve $B_4^{(2)}$ (solid line, computed by the methods of \cite{FikMavr}), together with local level curve (dashed line) associated with the critical point $\rho_c=1+2i$. Also shown is the corresponding cusp bisector (dotted line).
} \label{fig:Fig2}
\end{figure}
\begin{example} Let $A_n(\rho)=K_n(\rho)$.
For $n=4$ and for the type-2 case, one of the critical values predicted by Theorem \ref{th:preliminary-t} is $\rho_c=1+2i$, corresponding to $t_c=(1+i)/2$. As we know $n$ and $t_c$, we can calculate $a$ and $b$ using Theorem \ref{th:gamma-delta} (see below), and it turns out that $a=(-2+i)/\sqrt{2}$ and $b=1-i3/2$. Knowledge of $a$~and $b$ now allows us to apply Theorem \ref{th:local-level-curve}, and to thus obtain the local level curve shown in Fig. \ref{fig:Fig2} as the dashed lines. Similarly, Corollary \ref{corollary:cusp-bisector} gives  the cusp bisector, shown as the dotted line in Fig. \ref{fig:Fig2}. The solid line in Fig. \ref{fig:Fig2} is the borderline curve $B_4^{(2)}$, computed by the methods of \cite{FikMavr}. 

Fig. \ref{fig:Fig3} is like Fig. \ref{fig:Fig2}, but for $n=8$ and for a type-1 critical $\rho_c$ with $\rho_c\cong 0.922 - i1.29$. Note that the borderline curve $B_8^{(1)}$ exhibits more variations than does the $B_4^{(2)}$ of Fig. \ref{fig:Fig2}, in which we see better agreement between the borderline and local level curves.
\end{example}
\begin{figure}
\begin{center}
\includegraphics[scale=0.4]{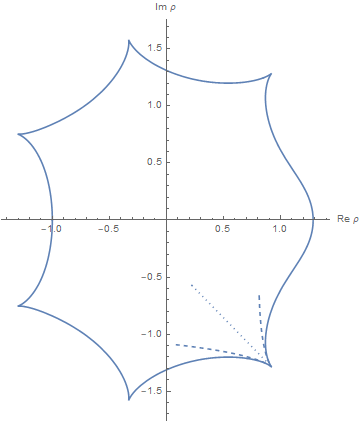}
\end{center}
\caption{Like Fig. \ref{fig:Fig2}, but for $n=8$ and for the  type-1 borderline curve $B_8^{(1)}$. Here, the type-1 critical $\rho_c$ is $\rho_c\cong 0.922 - i1.29$, corresponding to $t_c\cong 0.611-i0.274$ (see Theorem \ref{th:preliminary-t}). With these values of $n$, $\rho_c$, and $t_c$, we can use Theorem \ref{th:gamma-delta} to find $a\cong 2.74+i1.15$ and $b\cong 4.07+i4.52$, allowing us to compute the corresponding to $\rho_c$ local level curve (dashed line) and the cusp bisector (dotted line) via Theorem \ref{th:local-level-curve} and Corollary \ref{corollary:cusp-bisector}.
} \label{fig:Fig3}
\end{figure}
\subsection{Movement along the cusp bisector}

Imagine moving toward the cusp along the cusp bisector $\theta=\mathrm{arg}(\epsilon)=\pi-2\theta_a$, arriving at the cusp, and continuing in a straight line along the ray $\theta=\mathrm{arg}(\epsilon)=-2\theta_a$. We thus move along points $\rho$ given by
\begin{equation}
\label{eq:d}
\rho=\rho_c +d\exp(-i2\theta_a),\quad d\in\mathbb{R},
\end{equation}
with $d<0$ ($d>0$) before (after) the cusp. The magnitude $|d|=|\epsilon|=|\rho-\rho_c|$ is the distance to (from) the cusp. The theorem that follows gives small-$|d|$ formulas for the real parts, imaginary parts, and magnitudes of the two normalized eigenvalues.
\begin{figure}
\begin{center}
\includegraphics[scale=0.4]{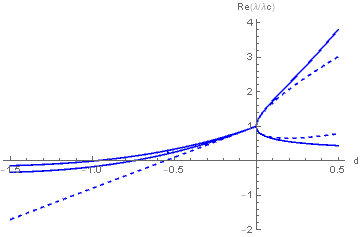}
\includegraphics[scale=0.4]{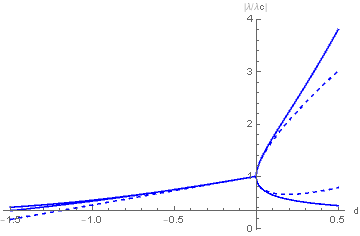}
\end{center}
\caption{Real parts (left) and magnitudes (right) of normalized eigenvalues along cusp bisector as function of $d$, for $n=4$ and $\rho_c=1+2i$ (i.e., for case corresponding to Fig. \ref{fig:Fig2}). The solid lines can be considered exact, while the dashed lines have been computed via Theorem \ref{th:evalues-in-terms-of-d}.} 
\label{fig:Fig4}
\end{figure}

\begin{theorem}
\label{th:evalues-in-terms-of-d}
Let $A_n(\rho)$ be an analytic matrix function that is subject to the restrictions and definitions of Section \ref{section:preamble}. Let $\rho$ be given by (\ref{eq:d}). Then as $d\to 0$ we have
\begin{equation}\label{eq:reallambda}
\mathrm{Re}\left\{\frac{\lambda (\rho)}{\lambda_c} \right\}=
\begin{cases} 1-|d||b|\cos\Theta+O\left(|d|^{3/2}\right),\ d\le 0,\\  
1\pm |a|\sqrt{d}+d|b|\cos\Theta+O\left(d^{3/2}\right),\ d\ge 0;
\end{cases}
\end{equation}
\begin{equation}\label{eq:imaglambda}
\mathrm{Im}\left\{\frac{\lambda (\rho)}{\lambda_c} \right\}=
\begin{cases}\pm |a|\sqrt{|d|}-|d||b|\sin\Theta+O\left(|d|^{3/2}\right),\ d\le 0,\\  
d|b|\sin\Theta+O\left(d^{3/2}\right),\ d\ge 0;
\end{cases}
\end{equation}
\begin{equation}
\label{eq:maglambda}
\left|\frac{\lambda (\rho)}{\lambda_c} \right|= 
\begin{cases} 1+|d|\,\frac{1}{2}\left(|a|^2-2|b|\cos\Theta\right)+O\left(|d|^{3/2}\right),\ d\le 0,\\  
1\pm |a|\sqrt{d}+d|b|\cos\Theta+O(d^{3/2}),\ d\ge 0.
\end{cases}
\end{equation}
\end{theorem}

\begin{proof}
By (\ref{eq:d}) and its surrounding discussions, we have $|\epsilon|=|d|$, $\theta=\pi-2\theta_a$ ($d<0$), and $\theta=-2\theta_a$ ($d>0$). Therefore (\ref{eq:approxlambda}) yields
\begin{equation*}
\frac{\lambda (\rho)}{\lambda_c}=
\begin{cases} 1\pm i|a|\sqrt{|d|}-|d||b|\exp(i\Theta)+O\left(|d|^{3/2}\right),\ d\le 0,\\  
1\pm |a|\sqrt{d}+|b|d\exp(i\Theta)+O\left(d^{3/2}\right),\ d\ge 0,
\end{cases}
\end{equation*}
where we used the last equation in (\ref{eq:notation}). Now calculate real parts, imaginary parts, and magnitudes to obtain (\ref{eq:reallambda})--(\ref{eq:maglambda}). 
\end{proof}

\begin{example} Figs. \ref{fig:Fig4} and \ref{fig:Fig5}, which correspond to Figs. \ref{fig:Fig2} and \ref{fig:Fig3}, respectively, show normalized real parts and magnitudes along our straight-line trajectory. Note that the horizontal scales in Figs. \ref{fig:Fig4} and \ref{fig:Fig5} are different. The solid lines have been computed numerically and can be considered to be exact. For $d<0$, there are two solid lines  in all cases (but the two are indistinguishable in the right Fig. \ref{fig:Fig5}). The dashed lines are the approximations found via Theorem \ref{th:evalues-in-terms-of-d}. The agreements are quite good (especially for $d<0$), and improve as we approach $d=0$. 
\end{example}
\begin{remark}
In Figs. \ref{fig:Fig4} and \ref{fig:Fig5} (as well as in Fig. \ref{fig:Fig7} below), for $d>0$ our approximate formulas underestimate (overestimate) the larger (smaller) of the two bifurcated normalized eigenvalues. By Section~6.3 of \cite{FikMavr}, the large-$|\rho|$ behaviors of the corresponding magnitudes are $|\rho|^{n-1}/n$ and $1/n$, respectively. These facts suggest that, soon after bifurcation,  the two eigenvalues rapidly head to their final (large-$|\rho|$) values. We corroborated this by many numerical results. 
\end{remark}
\begin{figure}
\begin{center}
\includegraphics[scale=0.4]{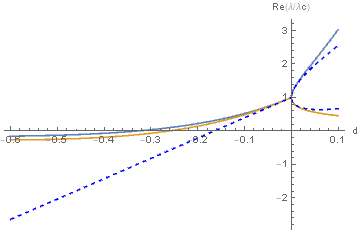}
\includegraphics[scale=0.4]{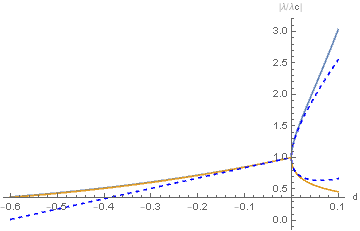}
\end{center}
\caption{Real parts (left) and magnitudes (right) of normalized eigenvalues along cusp bisector as function of $d$, for $n=8$ and $\rho_c=0.922-i1.29$ (i.e., for case corresponding to Fig. \ref{fig:Fig3}). The solid lines can be considered exact, while the dashed lines have been computed via Theorem \ref{th:evalues-in-terms-of-d}.} 
\label{fig:Fig5}
\end{figure}
\subsection{Consequences}
\label{section:consequences}

For a matrix $A_n(\rho)$ as described in Section \ref{section:preamble}, Theorems \ref{th:local-level-curve} and \ref{th:evalues-in-terms-of-d} have the following implications:
\\
\\
(i) Roughly speaking, Theorem \ref{th:evalues-in-terms-of-d} shows that, upon crossing the cusp (along our straight-line trajectory), the two  \textit{almost complex-conjugate} normalized eigenvalues become two  \textit{almost real} normalized eigenvalues; that, after crossing,
the real parts abruptly split and exhibit square-root behaviors, with one real part becoming larger than $1$ and the other smaller than $1$; and that the same is true of the magnitudes. More precisely, some of the statements just made are true to leading order. For example, before crossing the two imaginary parts do not differ by an overall sign, but their \textit{leading terms} so do; after crossing, these imaginary parts are not zero, but $O(d)$, etc.
\\
\\
(ii) Consider the magnitudes before bifurcation, as given by the top line ($d\le 0$) of  (\ref{eq:maglambda}). The coefficient of the $O(|d|)$ term is
\begin{equation}
\label{eq:parameter-c}
c\equiv\frac{1}{2}\left(|a|^2-2|b|\cos\Theta\right),
\end{equation}
which is nonzero if (\ref{eq:condition}) is assumed. When $c>0$ ($c<0$), the normalized magnitudes decrease (increase) before they reach the value $1$ and bifurcate. 
A smaller $|c|$ means a smaller slope before bifurcation. 
It also means a larger cusp-singularity coefficient in the cardioid equation (\ref{eq:cardioid}), so that the cusp bisector becomes more tightly sandwiched between the two branches of the local level curve. Therefore, it is not surprising that the magnitudes exhibit smaller slopes and less variation.
\\
\\
(iii)
Suppose that, upon arriving at the cusp, we \textit{turn} rather than continue along a straight line. As long as we do not turn back toward the path we came from, eigenvalue bifurcation still occurs, with the magnitudes and real parts still exhibiting a square-root behavior. But now, the coefficient of the $O\left(|\epsilon|^{1/2}\right)$ term in (\ref{eq:approxlambdamag2}) is less than $1$ (which is the coefficient value for the straight-line trajectory after the cusp, for which $\theta=-2\theta_a$). In this sense, there is ``less bifurcation'' than in the straight-line case. Furthermore, Theorem \ref{th:evalues-in-terms-of-d} observes bifurcations along the direction of ``maximum bifurcation.''

\section{On the Puiseux series of an analytic matrix function}
\label{section:puiseux-for-analytic}

For an analytic matrix function $A_n(\rho)$, Section \ref{section:analytic-matrix-functions} gives us an \textit{a priori} picture of bifurcation and related concepts near a critical point $\rho_c$ that corresponds to a double eigenvalue $\lambda_c$. To exploit the results of Section \ref{section:analytic-matrix-functions} quantitatively (as we have already done in our figures, which take $A_n(\rho)=K_n(\rho)$), it is necessary to know the Puiseux-series parameters. Accordingly, the present section develops formulas for $a$ and $b$, which will be specialized to $K_n(\rho)$ in the next section. We begin with two lemmas pertaining to inverse-function expansions. 

\begin{lemma}
\label{lemma:first-expansion-lemma}
Let $w(z)$ be a function that is analytic at $z=0$ with $w(0)=0$ and
\begin{equation}
\label{eq:function-w-expansion-1}
w(z)=\alpha_1z+\alpha_2z^2+O(z^3),\quad \mathrm{as} \quad z\rightarrow 0, 
\end{equation}
in which $\alpha_1\ne 0$. Then the inverse $z(w)$ exists and is analytic in some region containing $w=0$, with
\begin{equation}
\label{eq:function-z-expansion-1}
z(w)=\frac{1}{\alpha_1}w-\frac{\alpha_2}{\alpha_1^3}w^2+O(w^3),\quad \mathrm{as}\quad  w\rightarrow 0. 
\end{equation} 
\end{lemma}

\begin{proof}
The existence and analyticity of $z(w)$ follows from the inverse-function theorem (in its version for complex-analytic functions, see p. 74 of \cite{Ahlfors}). The coefficients in  (\ref{eq:function-z-expansion-1}) follow from those in  (\ref{eq:function-w-expansion-1}) by the usual inversion-of-series procedure (also called series reversion), see p. 183 of \cite{Ahlfors}, or see  \cite{MorseFeshbach}.
\end{proof}

Our next lemma gives an inverse-function expansion, analogous to the one in (\ref{eq:function-z-expansion-1}), for the case where the first derivative $w^{\prime}(z)$ vanishes. The said expansion is a Puiseux series rather than a power series.
\begin{lemma}
\label{lemma:second-expansion-lemma}
Let $w(z)$ be a function that is analytic at $z=0$, with $w(0)=0$,  $w^{\prime}(0)=0$, and
\begin{equation}
\label{eq:function-w-expansion-2}
w(z)=\alpha_2z^2+\alpha_3z^3+O(z^4),\quad \mathrm{as} \quad z\rightarrow 0, 
\end{equation}
in which $\alpha_2\ne 0$. Then, in some region containing $w=0$, $w(z)$ has a double-valued inverse $z(w)$. This inverse is an analytic function of the variable $\sqrt{w}$, with a Puiseux series given by
\begin{equation}
\label{eq:function-z-expansion-2}
z(w)=\frac{1}{\sqrt{\alpha_2}}\sqrt{w}-\frac{\alpha_3}{2\alpha_2^2}w+O(w^{3/2}),\quad \mathrm{as}\quad  w\rightarrow 0. 
\end{equation} 
\end{lemma}

\begin{proof}
Take the (double-valued) square root of (\ref{eq:function-w-expansion-2}) to obtain
 \begin{equation}
 \label{eq:temp1}
\sqrt{w(z)}=\sqrt{\alpha_2}\,z+\frac{\alpha_3}{2\sqrt{\alpha_2}}\,z^2+O(z^3),\quad \mathrm{as} \quad z\rightarrow 0, 
\end{equation}
which has a RHS like that of (\ref{eq:function-w-expansion-1}) (the two values $\sqrt{w}$ are accounted for by the two values of $\sqrt{\alpha_2}$). Therefore, Lemma \ref{lemma:first-expansion-lemma}, with $\sqrt{w}$ in place of $w$, can be applied to (\ref{eq:temp1}). This leads to (\ref{eq:function-z-expansion-2}).
\end{proof}

Suppose that we are given a function of $z$ like the RHS of (\ref{eq:function-w-expansion-1}), and another function of $z$ like the RHS of (\ref{eq:function-w-expansion-2}). The lemma that follows allows us to relate the function coefficients explicitly by eliminating the expansion parameter $z$.

\begin{lemma}
\label{lemma:f-in-terms-of-g}
Let $f(z)$ and $g(z)$ be two analytic functions with
\begin{equation}
\label{eq:f-expansion}
f(z)=f_0+f_1z+f_2z^2+O(z^3), \quad\mathrm{as}\quad z\to 0
\end{equation}
and
\begin{equation}
\label{eq:g-expansion}
g(z)=g_0+g_2z^2+g_3z^3+
O(z^4), \quad\mathrm{as}\quad z\to 0,
\end{equation}
in which $f_1\ne0$ and $g_2\ne 0$ (note that there is no $O(z)$ term in (\ref{eq:g-expansion})). Then $f$, regarded as a function of $g$, is analytic in the variable $(g-g_0)^{1/2}$, with a Puiseux series given by
\begin{equation}
\label{eq:f-g-expansion}
f=f_0+\frac{f_1}{\sqrt{g_2}}(g-g_0)^{1/2}+\frac{2f_2g_2-f_1g_3}{2g_2^2}(g-g_0)+O\left((g-g_0)^{3/2}\right),\quad\mathrm{as}\quad g\to g_0.
\end{equation}
\end{lemma}

\begin{proof}
By (\ref{eq:f-expansion}), the function $w=f-f_0$ has a RHS like that of (\ref{eq:function-w-expansion-1}). Therefore, Lemma \ref{lemma:first-expansion-lemma} can be applied to yield
\begin{equation*}
z=\frac{1}{f_1}(f-f_0)-\frac{f_2}{f_1^3}(f-f_0)^2+O\left((f-f_0)^3\right),\quad\mathrm{as}\quad f-f_0\to 0.
\end{equation*}
Substitute this  value of $z$ into (\ref{eq:g-expansion}) and collect like powers of $f-f_0$ to obtain
\begin{equation*}
g=g_0+\frac{g_2}{f_1^2}(f-f_0)^{2}+\frac{f_1g_3-2f_2g_2}{f_1^4}(f-f_0)^3+O\left((f-f_0)^{4}\right),\quad\mathrm{as}\quad f-f_0\to 0.
\end{equation*}
The equation just obtained shows that $w=g-g_0$ has a RHS like that of (\ref{eq:function-w-expansion-2}), with $f-f_0$ in place of $z$. Applying Lemma \ref{lemma:second-expansion-lemma} to this function $w$, we arrive at the desired conclusions.
\end{proof}

The aforementioned formulas for $a$ and $b$ are contained in the theorem that follows, which also contains sufficient conditions for the Puiseux series to exist.

\begin{theorem} 
\label{th:analytic-in-terms-of-mu} Let  $A_n(\rho)$ be an analytic matrix function and let $\rho(\mu)$ and $\lambda(\mu)$ be analytic functions in a region $D$, such that $A_n(\rho(\mu))$ has the eigenvalue $\lambda(\mu)$ for all $\mu\in D$.
Let $A_n(\rho_c)=A_n\left(\rho(\mu_c)\right)$ possess the double eigenvalue $\lambda_c=\lambda(\mu_c)$, where $\mu_c\in D$ and $\rho_c,\,\lambda_c\in\mathbb{C}$. Denote $\rho^\prime_c=\rho^\prime(\mu_c)$, $\lambda_c^{\prime\prime}=\frac{d^2\lambda}{d\mu^2}(\mu_c)$, etc., and assume that
\begin{equation}
\label{eq:rho-der-vanishes}
\lambda^{\prime}_c\ne0, \quad \rho^\prime_c=0, \quad \rho^{\prime\prime}_c\ne 0. 
\end{equation}
Then the double-valued function $\lambda(\rho)$ possesses a Puiseux series about the point $\rho=\rho_c$. This series is given by (\ref{eq:puiseux-original}) with
\begin{equation}
\label{eq:gamma-in-terms-of-mu-original}
\alpha=\lambda_c^\prime\sqrt{\frac{2}{\rho_c^{\prime\prime}}}\ne 0
\end{equation}
(in which the choice of square root is arbitrary) and
\begin{equation}
\label{eq:delta-in-terms-of-mu-original}
\beta=\frac{3\lambda_c^{\prime\prime}\rho_c^{\prime\prime}-\lambda_c^{\prime}\rho_c^{\prime\prime\prime}}{3(\rho_c^{\prime\prime})^2}.
\end{equation}
\end{theorem}

\begin{proof}
By analyticity and (\ref{eq:rho-der-vanishes}) we have 
\begin{equation*}
\label{eq:lambda-expansion}
\lambda(\mu)=\lambda_c+\lambda_c^{\prime}(\mu-\mu_c)+\frac{1}{2}\lambda_c^{\prime\prime}(\mu-\mu_c)^2+O\left((\mu-\mu_c)^3\right)
\end{equation*}
and
\begin{equation*}
\rho(\mu)=\rho_c+\frac{1}{2}\rho_c^{\prime\prime}(\mu-\mu_c)^2+\frac{1}{6}\rho_c^{\prime\prime\prime}(\mu-\mu_c)^3+O\left((\mu-\mu_c)^4\right).
\end{equation*}
These two equations have RHSs like those of (\ref{eq:f-expansion}) and (\ref{eq:g-expansion}), respectively, with $\mu-\mu_c$ in place of $z$. Consequently, we can eliminate $\mu-\mu_c$ by applying Lemma \ref{lemma:f-in-terms-of-g}. In this manner, we obtain the Puiseux series
\begin{equation*}
\lambda=\lambda_c+\sqrt{\frac{2}{\rho_c^{\prime\prime}}}\,\lambda_c^{\prime}(\rho-\rho_c)^{1/2}+\frac{3\lambda_c^{\prime\prime}\rho_c^{\prime\prime}-\lambda_c^{\prime}\rho_c^{\prime\prime\prime}}{3(\rho_c^{\prime\prime})^2}(\rho-\rho_c)+O\left((\rho-\rho_c)^{3/2}\right).
\end{equation*}
Comparison with (\ref{eq:puiseux-original}) gives (\ref{eq:gamma-in-terms-of-mu-original}) and (\ref{eq:delta-in-terms-of-mu-original}). Because of the $\pm$ sign in (\ref{eq:puiseux-original}), we can choose either square root in (\ref{eq:gamma-in-terms-of-mu-original}).
\end{proof}
\begin{corollary}
\label{corollary:analytic-in-terms-of-mu}
Assume, additionally, that $\lambda_c\ne 0$. Then $A_n(\rho)$ possesses a Puiseux series of the form (\ref{eq:puiseux}), with
\begin{equation}
\label{eq:gamma-in-terms-of-mu}
a=\frac{\lambda_c^\prime}{\lambda_c}\sqrt{\frac{2}{\rho_c^{\prime\prime}}}\ne 0
\end{equation}
(in which the choice of square root is arbitrary) and
\begin{equation}
\label{eq:delta-in-terms-of-mu}
b=\frac{3\lambda_c^{\prime\prime}\rho_c^{\prime\prime}-\lambda_c^{\prime}\rho_c^{\prime\prime\prime}}{3\lambda_c(\rho_c^{\prime\prime})^2}.
\end{equation}
\end{corollary}

\section{Puiseux-series parameters for the Kac--Murdock--Szeg\H{o} matrix}
\label{section:puiseux-for-KMS}

This section is a straightforward application of Theorem \ref{th:analytic-in-terms-of-mu}/Corollary \ref{corollary:analytic-in-terms-of-mu} for the case $A_n(\rho)=K_n(\rho)$, in which the analytic functions $\lambda(\mu)$ and $\rho(\mu)$ of Theorem \ref{th:analytic-in-terms-of-mu} are given in (\ref{eq:lambda-of-mu}) and (\ref{eq:rho-of-mu}). The results are explicit expressions for $a$ and $b$ in terms of $n$ and the parameter $t_c$ of Theorem \ref{th:preliminary-t}. We specifically prove the following theorem.

\begin{theorem} 
\label{th:gamma-delta}
Let $K_n(\rho)$ ($n\ge 3$) possess the (type-1 or type-2) double eigenvalue $\lambda_c=-n$ when $\rho=\rho_c\in\mathbb{C}$, corresponding to $t=t_c\in\mathbb{C}$ (see  Theorem \ref{th:preliminary-t}). For this specific value of $\rho_c$, let $a$ and $b$ be the two Puiseux-series parameters defined via (\ref{eq:puiseux}). Then $a$ and $b$
can be found from the equations
\begin{equation}
\label{eq:gamma-in-terms-of-t}
a=i\sqrt{\frac{2}{n}}\left[\frac{(t_c\typemp  T_n(t_c))(1\typemp T_{n-1}(t_c))}{1-t_c^2}\right]^{1/2}
\end{equation}
(in which the choice of square root is arbitrary) and
\begin{multline}
\label{eq:delta-in-terms-of-t-new-expression}
\left[6n(t_c^2-1)(t_c\typemp T_n(t_c))\right]b=12t_c^2+5n(n+1)(t_c^2-1)+4(n+1)T_{n-1}^2(t_c)-4(n-2)T_n^2(t_c)
\\
\typemp(n+1)(4t_c^2-1)T_{n-1}(t_c)\typepm 3(n-7)t_cT_n(t_c),
\end{multline}
where the meaning of the symbols $\typepm$ and $\typemp$ is discussed in Section \ref{section:notation}.
\end{theorem}

\begin{proof}  By (\ref{eq:rho-prime-in-terms-of-mu}) and (\ref{eq:lambdac-equals-minus-n}) we have
\begin{equation}
\label{eq:rho-der-vanishes-kms}
\rho'(\mu_c)=\rho'_c=0.
\end{equation}
We can thus apply Corollary \ref{corollary:analytic-in-terms-of-mu} for the case $A_n(\rho)=K_n(\rho)$, for which $\mu_c$ is given via (\ref{eq:tc-mc}).
Our proof consists of calculating the derivatives appearing in (\ref{eq:gamma-in-terms-of-mu}) and (\ref{eq:delta-in-terms-of-mu}), and then expressing the resulting formulas for $a$ and $b$ in terms of $t_c=\cos \mu_c$. 

Eqn. (\ref{eq:lambda-of-mu}) gives 
\begin{equation}
\label{eq:lambda-to-differentiate}
\typemp\sin(n\mu)=\lambda(\mu)\sin\mu.
\end{equation}
Take the derivative of (\ref{eq:lambda-to-differentiate}) with respect to $\mu$, set $\mu=\mu_c$, replace $\lambda(\mu_c)=\lambda_c$ using (\ref{eq:lambdac-equals-minus-n}), and solve for $\lambda^\prime_c= \lambda^\prime(\mu_c)$ to obtain
\begin{equation}
\label{eq:lambda-der}
\lambda_c^\prime=\frac{n[\cos\mu_c\typemp \cos(n\mu_c)]}{\sin\mu_c},
\end{equation}
which gives $\lambda^\prime_c$ in terms of $n$ and $\mu_c$.

Now repeat the above procedure taking \textit{two} derivatives of (\ref{eq:lambda-to-differentiate}), and solve for  $\lambda^{\prime\prime}_c=\lambda^{\prime\prime}(\mu_c)$ to obtain 
\begin{equation}
\label{eq:lambda-der-der}
\lambda^{\prime\prime}_c=\frac{-n\sin\mu_c-2\lambda^\prime_c\cos\mu_c\typepm n^2 \sin(n\mu_c)}{\sin\mu_c},
\end{equation}
which gives  $\lambda^{\prime\prime}_c$ in terms of $n$, $\mu_c$, and $\lambda^\prime_c$ ($\lambda^\prime_c$ has already been found in  (\ref{eq:lambda-der})).

Eqn. (\ref{eq:rho-prime-in-terms-of-mu}) gives
\begin{equation}
\label{eq:rho-to-differentiate}
\rho^\prime(\mu)[1\typemp \cos((n-1)\mu)]=- (\lambda(\mu)+n)\sin\mu.
\end{equation}
Take the derivative of (\ref{eq:rho-to-differentiate}) with respect to $\mu$, set $\mu=\mu_c$, replace $\lambda(\mu_c)=\lambda_c$ using (\ref{eq:lambdac-equals-minus-n}), replace $\rho^\prime(\mu_c)=\rho^\prime_c=0$ using (\ref{eq:rho-der-vanishes-kms}),  and solve for $\rho^{\prime\prime}_c= \rho^{\prime\prime}(\mu_c)$ to obtain
\begin{equation}
\label{eq:rho-der-der}
\rho^{\prime\prime}_c=\frac{-\lambda^\prime_c\sin\mu_c}{1\typemp \cos[(n-1)\mu_c]},
\end{equation}
in which $\lambda^\prime_c$ has been found in (\ref{eq:lambda-der}).

Repeat the above procedure taking \textit{two} derivatives of (\ref{eq:rho-to-differentiate}), and solve for  $\rho^{\prime\prime\prime}_c=\rho^{\prime\prime\prime}(\mu_c)$ to obtain
\begin{equation}
\label{eq:rho-der-der-der}
\rho^{\prime\prime\prime}_c=\frac{\typemp2(n-1)\rho^{\prime\prime}_c\sin[(n-1)\mu_c]-\lambda^{\prime\prime}_c\sin\mu_c-2\lambda^\prime_c\cos\mu_c}{1\typemp \cos[(n-1)\mu_c]},
\end{equation}
which is to be used together with (\ref{eq:lambda-der}), (\ref{eq:lambda-der-der}), and (\ref{eq:rho-der-der}).

Now substitute (\ref{eq:lambdac-equals-minus-n}), (\ref{eq:lambda-der}), (\ref{eq:lambda-der-der}), (\ref{eq:rho-der-der}), and (\ref{eq:rho-der-der-der}) into (\ref{eq:gamma-in-terms-of-mu}) and (\ref{eq:delta-in-terms-of-mu}). Application of trigonometric identities (this step is cumbersome but straightforward) leads to
\begin{equation*}
a=i\sqrt{\frac{2}{n}}\sqrt{\frac{[\cos\mu_c\typemp\cos(n\mu_c)][1\typemp \cos((n-1)\mu_c)]}{\sin^2\mu_c}}
\end{equation*}
and
\begin{multline*}
-6n\sin^2\mu_c[\cos\mu_c\typemp\cos(n\mu_c)]\,b\\=12+6\cos(2\mu_c)+2(n+1)\cos[2(n-1)\mu_c]-2(n-2)\cos(2n\mu_c)\\
\typemp (n+1)\cos[(n-3)\mu_c]
\typepm(n-23)\cos\mu_c\cos(n\mu_c)
\typemp 5(n+1)\sin\mu_c\sin(n\mu_c).
\end{multline*}
The desired equation (\ref{eq:gamma-in-terms-of-t}) results from the equation for $a$ just obtained via (\ref{eq:chebyshev}) and (\ref{eq:tc-mc}). Similarly, from the equation for $b$ just obtained, via (\ref{eq:chebyshev}) and (\ref{eq:tc-mc}) we first get
\begin{multline}
\label{eq:delta-in-terms-of-t-old-expression}
\left[6n(t_c^2-1)(t_c\typemp T_n(t_c))\right]b=6+12t_c^2+5n(n+1)(t_c^2-1)+2(n+1)T_{2n-2}(t_c)\\-2(n-2)T_{2n}(t_c)\typemp(n+1)T_{n-3}(t_c)\typepm(n-23)t_c\,T_n(t_c),
\end{multline}
in which we have used the equality
\begin{equation*}
\typemp 5(n+1)\sin\mu_c\sin(n\mu_c)=\typemp 5(n+1)(1-t_c^2)U_{n-1}(t_c)=5n(n+1)(t_c^2-1),
\end{equation*}
which follows from (\ref{eq:chebyshev}), (\ref{eq:tc-mc}), and Theorem \ref{th:preliminary-t}. Finally, (\ref{eq:delta-in-terms-of-t-new-expression}) results from (\ref{eq:delta-in-terms-of-t-old-expression}) if we substitute $T_{2n-2}(t_c)$ and $T_{2n}(t_c)$ using (\ref{eq:b7}), and express $T_{n-3}(t_c)$ in terms of $T_{n-1}(t_c)$ and $T_{n}(t_c)$ by twice using (\ref{eq:b8}).
\end{proof}

We have already given qualitative and quantitative results based on the formulas in Theorem \ref{th:gamma-delta}. While suitable for numerical calculations, those formulas are complicated. We now proceed to a case where  expressions simplify and become more informative. In the case to be discussed, we will also see that the bifurcating eigenvalues are \textit{exactly} real.

\section{Kac--Murdock--Szeg\H{o}  matrix with a purely imaginary $\rho_c$}
\label{section:purely-imaginary-rho}

This section deals with Kac--Murdock--Szeg\H{o}  matrices $K_n(\rho)$ that satisfy the following restrictions.\\ 
(i) The matrix dimension $n$ takes on the values $n=3,5,7,9,\ldots$.\\
(ii) The parameter $\rho$ is close to a purely imaginary critical value $\rho_c=iy_{n}$ ($y_{n}>0$), for which $K_n(\rho)$ possesses a double eigenvalue (equal, as always, to $-n$). We will soon verify that such a $\rho_c$ indeed exists, and show how to compute it.

When we observe the two eigenvalues, we assume a movement along the imaginary-$\rho$ axis  and close to $y_{n}$; as we will see, the imaginary axis underlies the cusp bisector. We thus track the variations of $\lambda(\rho)=\lambda(iy)$, where $y\in\mathbb{R}$ and  $y\to y_{n}$.

Subject to the aforesaid conditions, Theorem \ref{th:evalues-in-terms-of-y} will demonstrate that our major results simplify considerably. In particular, it turns out that $\theta_a=3\pi/4$, $\theta_b=-\pi/2$, and $\Theta=-2\pi$. It is thus helpful to alter our notation somewhat, and to denote $|a|$, $|b|$, and $c$ by $a_n$, $b_n$, and $c_n$ respectively (but we retain our notation for $d$). 

Theorem \ref{th:evalues-in-terms-of-y} will be conveniently expressed in terms of a new $n$-dependent parameter that we denote by $x_n$. It is defined as follows.

\begin{lemma}
\label{lemma:x-parameter}
For $n=3,5,7,9\ldots$, the equation
\begin{equation}
\label{eq:x-parameter}
T_n(x)=nx,\quad x\in(1,+\infty)
\end{equation}
has a unique solution $x=x_n$.
\end{lemma}

\begin{proof}
By (\ref{eq:b2})--(\ref{eq:b3}), the function $f_n(x)=T_n(x)-nx$ satisfies $f_n(1)=1-n<0$, $f_n(+\infty)=+\infty$, and has the positive derivative $f^\prime_n(x)=n[U_{n-1}(x)-1]$.
\end{proof}

Given $n$, we can find $x_n$ by numerically solving (\ref{eq:x-parameter}) for $x$. The lemma that follows provides an alternative method.

\begin{lemma}
\label{lemma:v-parameter}
The parameter $x_n$ defined in Lemma \ref{lemma:x-parameter} can also be found from
\begin{equation}
\label{eq:x-in-terms-of-v}
x_n=\cosh v_n,
\end{equation}
where $v=v_n$ is the unique positive solution to the equation\footnote{The symbol $v$ has been chosen judiciously because, for $n=3,5,\ldots$,  (\ref{eq:v-parameter}) is equivalent to (4.3) of \cite{FikMavr}, in the special case when $u=\pi/2$ in the latter equation. This is true because $\rho_c=iy_n$ is the point where the borderline curve $B_n^{(k)}$ intersects the positive-imaginary semi-axis.}
\begin{equation}
\label{eq:v-parameter}
\cosh(nv)=n\cosh v.
\end{equation}
\end{lemma}

\begin{proof}
In Lemma \ref{lemma:x-parameter}, set $x=\cosh v$ and use (\ref{eq:b101}). 
\end{proof}

Solving (\ref{eq:v-parameter}) instead of (\ref{eq:x-parameter}) may be advantageous for numerics,  especially when using the large-$n$ formula (\ref{eq:vn-large-n}) (see below) as an initial approximation for iterative root-finding routines. 

The numbers $x_n$ help us obtain purely imaginary zeros of the polynomials $q_1(n,t)$ and $q_2(n,t)$ that appear in Theorem \ref{th:preliminary-t}:

\begin{lemma}
\label{lemma:imaginary-zeros}
For $n=3,5,7,\ldots$, the $x_n$ defined in Lemma \ref{lemma:x-parameter} satisfy
\begin{equation}
U_{n-1}\left(i\sqrt{x_n^2-1}\right)=U_{n-1}\left(\sqrt{1-x_n^2}\right)=(-1)^\frac{n-1}{2}\,n.
\end{equation}
\end{lemma}
 
\begin{proof}
The first equality follows from (\ref{eq:b51}) and the second from (\ref{eq:b11}) and (\ref{eq:x-parameter}).
\end{proof}

The central result of this section is the following theorem, which specializes our main results so far.

\begin{theorem}
\label{th:evalues-in-terms-of-y}
For $n=3,5,7,9,\ldots$, let $y_{n}$ be the positive number
\begin{equation}
\label{eq:y-c}
y_{n}=\frac{1}{\sqrt{x_n^2-1}}\frac{T_{(n+1)/2}(x_n)}{U_{(n-3)/2}(x_n)},
\end{equation}
where $x_n$ is defined in Lemma \ref{lemma:x-parameter}, and let $\rho_c=iy_n$.
Then the matrix $K_n(\rho_c)=K_n(iy_{n})$ possesses the double eigenvalue $\lambda_c=-n$. This double eigenvalue is of type-1 when $n=5,9,13,17\ldots$ and of type-2 when $n=3,7,11,15\ldots$.

The associated Puiseux-series parameters $a_n\equiv |a|$, $b_n\equiv |b|$, $\theta_a$, $\theta_b$, and $\Theta$ (see (\ref{eq:puiseux}) and (\ref{eq:notation})) can be found from
\begin{equation}
\label{eq:thetas-for-imag}
\theta_a=\frac{3\pi}{4},\quad \theta_b=-\frac{\pi}{2},\quad \Theta=-2\pi,
\end{equation}
\begin{equation}
\label{eq:mag-gamma-for-imag}
a_{n}=\sqrt{\frac{2}{n}}\,\frac{(x_n^2-1)^{1/4}}{x_n}\sqrt{[U_{n-1}(x_n)-1][T_{n-1}(x_n)-1]},
\end{equation}
and
\begin{multline}
\label{eq:mag-delta-for-imag}
6nx_n^2\sqrt{x_n^2-1}[U_{n-1}(x_n)-1]b_{n}=12+4(n+1)T_{n-1}^2(x_n)-(5n^2+5n+12)x_n^2\\
-3(n-7)(x_n^2-1)U_{n-1}(x_n)
+4(n-2)(n^2x_n^2-1)
+(n+1)(4x_n^2-3)T_{n-1}(x_n).
\end{multline}

The associated cusp bisector is $\theta=-\pi/2$, which is the ray that starts from the point $\rho=iy_{n}$, lies on the imaginary axis, and points downward. Sufficiently close to this ray, the local level curve is given by the polar equation
\begin{equation}
\label{eq:near-cardioid-for-imag}
|\epsilon|=|\epsilon|(\theta)=\frac{8a_{n}^2(1+\sin\theta)}{[a_{n}^2+(4b_{n}-a_{n}^2)\sin\theta]^2}\quad (\epsilon=\rho-iy_{n}),
\end{equation}
and is symmetric with respect to the imaginary axis.

Finally, let $\rho=iy$ with $y>0$ and let $d=y-y_{n}$. Then as 
$y\to y_{n}$
we have
\begin{equation}\label{eq:reallambda-for-imag}
\mathrm{Re}\left\{\frac{\lambda (\rho)}{\lambda_c} \right\}=
\begin{cases} 1-|d|b_{n}+O\left(|d|^{3/2}\right),\ d=y-y_n\le 0,\\  
1\pm a_{n}\sqrt{d}+db_{n}+O\left(d^{3/2}\right),\ d=y-y_n\ge 0;
\end{cases}
\end{equation}
\begin{equation}\label{eq:imaglambda-for-imag}
\mathrm{Im}\left\{\frac{\lambda (\rho)}{\lambda_c} \right\}=
\begin{cases}\pm a_{n}\sqrt{|d|}+O\left(|d|^{3/2}\right),\ d=y-y_n\le 0,\\  
O\left(d^{3/2}\right),\ d=y-y_n\ge 0;
\end{cases}
\end{equation}
\begin{equation}\label{eq:maglambda-for-imag}
\left|\frac{\lambda (\rho)}{\lambda_c} \right|= 
\begin{cases} 1+|d|\,\frac{1}{2}\left(a_{n}^2-2b_{n}\right)+O\left(|d|^{3/2}\right),\ d=y-y_n\le 0,\\  
1\pm a_{n}\sqrt{d}+db_{n}+O(d^{3/2}),\ d=y-y_n\ge 0.
\end{cases}
\end{equation}
\end{theorem}

\begin{proof}
First let $n=5,9,13,\ldots$. By Lemma \ref{lemma:imaginary-zeros}, the positive-imaginary number $t_c$ given by
\begin{equation}
\label{eq:tc-for-imag}
t_c=i\sqrt{x_n^2-1}
\end{equation}
is a zero of the polynomial $q_1(n,t)$ of Theorem \ref{th:preliminary-t}. Consequently, $\lambda_c=-n$ is a double and type-1 eigenvalue of $K_n(\rho_c)=K_n(iy_{n})$ where, by (\ref{eq:rho-of-t-type-1}),
\begin{equation*}
\rho_c=iy_{n}=\frac{U_{(n-1)/2}\left(i\sqrt{x_n^2-1}\right)}{U_{(n-3)/2}\left(i\sqrt{x_n^2-1}\right)}.
\end{equation*}
Use of (\ref{eq:b11}) then gives the expression for $y_{n}$ in (\ref{eq:y-c}), in which $\sqrt{x_n^2-1}>0$. 

When $n=3,7,11,\ldots$,  the (same) $t_c$ of (\ref{eq:tc-for-imag}) is a zero of $q_2(n,t)$, so that $\lambda_c=-n$ is a double and type-2 eigenvalue of $K_n(\rho_c)=K_n(iy_{n})$, where 
\begin{equation*}
\rho_c=iy_{n}=\frac{T_{(n+1)/2}\left(i\sqrt{x_n^2-1}\right)}{T_{(n-1)/2}\left(i\sqrt{x_n^2-1}\right)}.
\end{equation*}
This time, use of (\ref{eq:b1}) gives the (same) expression (\ref{eq:y-c}). 

Substitute (\ref{eq:tc-for-imag}) into (\ref{eq:gamma-in-terms-of-t}) and use (\ref{eq:b1}) and $a_n\equiv |a|$ to obtain
\begin{equation*}
a=
\sqrt{\frac{2}{n}}\,\frac{(x_n^2-1)^{1/4}}{x_n}\sqrt{[U_{n-1}(x_n)-1][T_{n-1}(x_n)-1]}
\exp\left(i\frac{3\pi}{4}\right)
\end{equation*}
(the choice of the complex square root was arbitrary). By $x_n>1$ and the inequalities in (\ref{eq:b6}), the quantity multiplying $\exp(i3\pi/4)$ is positive. Thus, comparison with (\ref{eq:notation}) gives $\theta_a=3\pi/4$ and (\ref{eq:mag-gamma-for-imag}), where $a_n\equiv |a|$. 

Starting from (\ref{eq:delta-in-terms-of-t-new-expression}), and with $b_n\equiv |b|$, we can show (\ref{eq:mag-delta-for-imag}) and $\theta_b=-\pi/2$ in a similar manner: We first obtain
\begin{multline*}
i6nx_n^2\sqrt{x_n^2-1}[U_{n-1}(x_n)-1]b=12+4(n+1)T_{n-1}^2(x_n)-(5n^2+5n+12)x_n^2\\
-3(n-7)(x_n^2-1)U_{n-1}(x_n)+4(n-2)(x_n^2-1)U_{n-1}^2(x_n)+(n+1)(4x_n^2-3)T_{n-1}(x_n),
\end{multline*}
and then use the equality
\begin{equation}
\label{eq:u-n-1-first}
(x_n^2-1)U_{n-1}^2(x_n)=n^2x_n^2-1,
\end{equation}
which follows from (\ref{eq:b9}) and Lemma \ref{lemma:x-parameter}. Positivity of the RHS of (\ref{eq:mag-delta-for-imag}) was verified numerically and via the large-$n$ formula (\ref{eq:b-large-n}) (see below).

Condition (\ref{eq:condition}) was verified in the same way. Eqn. (\ref{eq:near-cardioid-for-imag}) is then an immediate consequence of (\ref{eq:thetas-for-imag}), (\ref{eq:near-cardioid}), $a_n\equiv |a|$, and $b_n\equiv |b|$; it describes a curve that is symmetric with respect to the imaginary axis.

With $\rho=iy$, $\rho_c=iy_n$, and $\theta_a=3\pi/4$, (\ref{eq:d}) gives $d=y-y_n$. Eqns.   (\ref{eq:reallambda-for-imag})--(\ref{eq:maglambda-for-imag}) then result from (\ref{eq:thetas-for-imag}),  (\ref{eq:reallambda})--(\ref{eq:maglambda}), $a_n\equiv |a|$, and $b_n\equiv |b|$. 
\end{proof}

The corollary that follows gives alternative formulas for $y_n$, $a_n$, and $b_n$ in terms of the parameter $v_n$.

\begin{corollary} 
\label{corollary:in-terms-of-vn}
Let $n=3,5,7,9,\ldots$ and let $v_n$ be the parameter defined in Lemma \ref{lemma:v-parameter}. The $y_n$ of (\ref{eq:y-c}) can also be found via $v_n$ from the equation
\begin{equation}
\label{eq:yn-in-terms-of-vn}
y_n=\frac{\cosh(\frac{n+1}{2}v_n)}{\sinh(\frac{n-1}{2}v_n)}.
\end{equation}
Furthermore, the quantities appearing in the RHSs of (\ref{eq:mag-gamma-for-imag}) and (\ref{eq:mag-delta-for-imag}) (i.e, in the RHSs of the formulas for $a_n$ and $b_n$) can also be found via $v_n$ from
\begin{equation}
\label{eq:xn-in-terms-of-vn}
x_n=\cosh v_n,
\end{equation}
\begin{equation}
\label{eq:sqrt-in-terms-of-vn}
\sqrt{x_n^2-1}=\sinh v_n,
\end{equation}
\begin{equation}
\label{eq:tt-in-terms-of-vn}
T_{n-1}(x_n)=\cosh[(n-1)v_n],
\end{equation}
\begin{equation}
\label{eq:uu-in-terms-of-vn}
U_{n-1}(x_n)=\frac{\sinh (nv_n)}{\sinh v_n}.
\end{equation}
\end{corollary}

\begin{proof}
Eqn. (\ref{eq:xn-in-terms-of-vn}) is the same as (\ref{eq:x-in-terms-of-v}) and immediately yields  (\ref{eq:sqrt-in-terms-of-vn}). Eqns. (\ref{eq:yn-in-terms-of-vn}), (\ref{eq:tt-in-terms-of-vn}) and (\ref{eq:uu-in-terms-of-vn}) follow from (\ref{eq:xn-in-terms-of-vn}), (\ref{eq:sqrt-in-terms-of-vn}), (\ref{eq:b101}), and  (\ref{eq:b102}).
\end{proof}
\begin{figure}
\begin{center}
\includegraphics[scale=0.6]{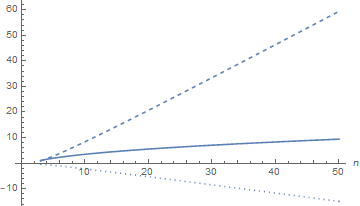}
\end{center}
\caption{Puiseux-series parameters $a_n$ (solid line), $b_n$ (dashed line), and $c_n$ (dotted line), as function of $n$ ($3\le n\le 50$), computed by the methods of Section \ref{section:purely-imaginary-rho}.} \label{fig:Fig6}
\end{figure}
Fig. \ref{fig:Fig6} shows $a_n$, $b_n$, and 
\begin{equation}
\label{eq:parameter-c-n}
c_n=\frac{a_n^2-2b_n}{2}
\end{equation}
($c_n$ is our new notation for the  parameter $c$ of (\ref{eq:parameter-c})) as a function of $n$, for $3\le n\le 50$.

\begin{example}
For the case $n=19$, Lemma \ref{lemma:v-parameter} gives $v_{19}\cong 0.192$. Corollary \ref{corollary:in-terms-of-vn}, (\ref{eq:mag-gamma-for-imag}), and (\ref{eq:mag-delta-for-imag}) then give $y_{19}\cong 1.28$, $a_{19}\cong 5.39$, and $b_{19}\cong 19.4$. With these values, eqn. (\ref{eq:maglambda-for-imag}) of Theorem   \ref{th:evalues-in-terms-of-y} produces the normalized magnitude variations shown in Fig. \ref{fig:Fig7} as the dashed lines. As before, the solid lines are the numerically computed eigenvalues and can be considered exact.
\begin{figure}
\begin{center}
\includegraphics[scale=0.4]{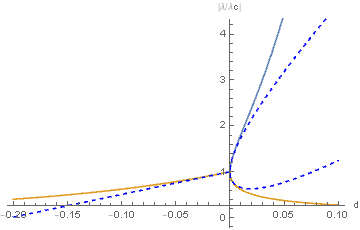}
\end{center}
\caption{Magnitudes of normalized eigenvalues along imaginary axis as function of $d=y-y_{19}\cong y-1.28$, for $n=19$. The solid lines can be considered exact, while the dashed lines have been computed via the formulas in Section \ref{section:purely-imaginary-rho}. Here, as opposed to Figs. \ref{fig:Fig4} and \ref{fig:Fig5}, there is only one solid line for $d<0$.}
\label{fig:Fig7}
\end{figure}
\begin{figure}
\begin{center}
\includegraphics[scale=0.5]{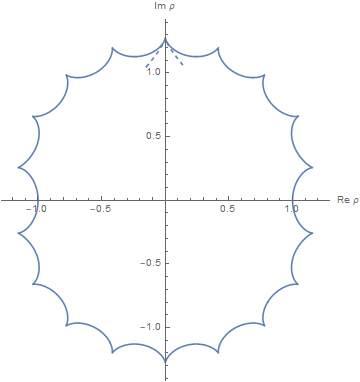}
\end{center}
\caption{Type-2 borderline curve $B_{19}^{(2)}$ (solid line), together with local level curve corresponding to $\rho_c=iy_{19}\cong i1.28$ (dashed line).
} \label{fig:Fig8}
\end{figure}
Eqn. (\ref{eq:near-cardioid-for-imag}) of Theorem \ref{th:evalues-in-terms-of-y} gives the local level curve shown in Fig. \ref{fig:Fig8} as the two dashed lines. The closed solid line in Fig. \ref{fig:Fig8} is the borderline curve $B_{19}^{(2)}$, computed by the methods of \cite{FikMavr}.  Note that $B_{19}^{(2)}$ has 18 cusps, corresponding to the $n-1=18$ zeros of the polynomial $q_2(n,t)=U_{n-1}(t)+n=U_{18}(t)+19$ that appears in Theorem \ref{th:preliminary-t}. 
\end{example}

\subsection{Consequences; extensions}
\label{section:consequences-KMS}

The heretofore results of Section \ref{section:purely-imaginary-rho} lead to further conclusions as follows.
\\
\\
(i) Corollary \ref{corollary:isotropic} and  Theorem \ref{th:eigenvectors} give 
the elements $v_{cj}$ of the  isotropic eigenvector $\textit{\textbf{v}}_c$ associated with $\lambda_c=-n$. When $\rho_c=iy_n$, we can work as we did above to express $v_{cj}$ in terms of the   parameter $x_n$. In the type-1 case ($n=5,9,13,\ldots$), the final result is
\begin{equation*}
v_{cj}=\mathrm{sgn}\left(\frac{n-1}{2}-j\right)\, U_{|(n-1-2j)/2|-1}\left(i\sqrt{x_n^2-1}\right),
\quad j=0,1,\ldots,n-1.
\end{equation*}
where $\mathrm{sgn}(x)=-1$, $0$, $1$ when $x<0$, $x=0$, $x>0$, respectively, where $U_{-1}(z)=0$, and where we omitted an overall factor. Similarly, the result pertaining to the type-2 case ($n=3,7,11,\ldots$) is
\begin{equation*}
v_{cj}=
T_{|(n-1-2j)/2|}\left(i\sqrt{x_n^2-1}\right),\quad j=0,1,\ldots,n-1.
\end{equation*}
In both cases, the vector elements $v_{cj}$ alternate between real and purely imaginary values.
\\
\\
(ii) We can further show that $-iy_{n}$ is also a critical point. It lies on the \textit{negative} imaginary axis and is of the same type (1 or 2) as  $\rho_c=+iy_{n}$. The imaginary axis bisects both associated cusps and is an axis of symmetry of both associated level curves, see the examples in Figs.~\ref{fig:Fig1} and \ref{fig:Fig8}. 
\\
\\
(iii) Observe from Fig. \ref{fig:Fig6} (and, also, from (\ref{eq:a-large-n}) below) that $a_n$ increases with $n$. This observation together with (\ref{eq:reallambda-for-imag}) and (\ref{eq:maglambda-for-imag}) imply that ``more bifurcation'' (in the sense specified in Conclusion (iii) of Section \ref{section:consequences})   occurs when $n$ becomes larger. Compare the examples in Figs. \ref{fig:Fig4}, \ref{fig:Fig5}, and \ref{fig:Fig7}.
\\
\\
(iv)
When we view (\ref{eq:reallambda-for-imag}) and (\ref{eq:imaglambda-for-imag}) in the light of Corollary \ref{corollary:symmetries} and the conditions $a_n>0$ and $b_n>0$ we understand that, upon crossing the cusp, the two exactly complex-conjugate eigenvalues (which initially have nonzero imaginary parts) become two exactly real eigenvalues. In other words, in the special case where the critical $\rho_c$ lies on the imaginary axis, Conclusion (i) of Section \ref{section:consequences} is \textit{exactly} true (not only to leading order). Corollary \ref{corollary:symmetries} also explains why Fig. \ref{fig:Fig7} only has \textit{one} solid line (not two) for $d<0$.
\\
\\
(v) As seen from Fig. \ref{fig:Fig6} (and, also, from the large-$n$ results of the next section), we always have $c_n<0$. Therefore, by Conclusion (ii) of Section \ref{section:consequences}, the normalized magnitudes $|\lambda_c(iy)/\lambda_c|$ always \textit{increase} before they reach the value $1$ and bifurcate; see the example in Fig. \ref{fig:Fig7}. This means that the magnitudes $|\lambda_c(iy)|$ increase, reach the value $|\lambda_c|=n$, and then bifurcate.
\\
\\
(vi) 
By Definition 1.1 of \cite{FikMavr}, an eigenvalue $\lambda$ of $K_n(\rho)$ is called \textsl{extraordinary} if $|\lambda|>n$. Conclusion (v) above tells us that, upon crossing the cusp, the number of  extraordinary type-$k$ eigenvalues increases by 1, where $k$ is the type of the associated borderline curve $B_n^{(k)}$. This conclusion is true, more generally, for any movement of the kind described in Conclusion (iii) of Section \ref{section:consequences}  (after arriving at the cusp, turn without heading backwards, and continue along a straight line): Such a movement, too, increases the number of type-$k$ extraordinary eigenvalues by 1. We note that Section 6.2 of \cite{FikMavr} arrives at a similar conclusion via different means---in fact, the conclusion there is conditional, as it is subject to the conjecture that the closed curve $B_n^{(k)}$ is a Jordan curve. 
\begin{figure}
\begin{center}
\includegraphics[scale=0.4]{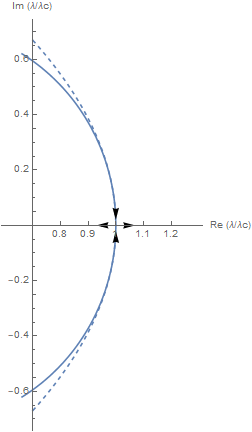}
\end{center}
\caption{Parabola of (\ref{eq:parabola}) (dashed line), together with actual eigenvalue trajectories (solid line), for $n=19$ (i.e., for the same case as in Figs. \ref{fig:Fig7} and \ref{fig:Fig8}). The arrows give the direction of movement of the two eigenvalues. The point $(1,0)$ is the collision (bifurcation) point. \label{fig:Fig9}}
\end{figure}
\\
\\
(vii) For $d=y-y_n\le 0$, let us denote $\chi\equiv 1-|d|b_n$ and $\psi\equiv\pm a_n\sqrt{|d|}$, so that $\chi$ consists of the two leading terms in the normalized  real parts before bifurcation, see the top line in (\ref{eq:reallambda-for-imag}). Similarly, $\psi$ is the leading term of the normalized imaginary part, see (\ref{eq:imaglambda-for-imag}). Elimination of $|d|$ gives
\begin{equation}
\label{eq:parabola}
\psi^2=\frac{a_n^2}{b_n}(1-\chi).
\end{equation} 
Eqn. (\ref{eq:parabola}), which is the equation of a parabola, gives the approximate trajectory in the normalized complex $\lambda$-plane. Before bifurcation, the two complex-conjugate normalized eigenvalues move along the parabola of (\ref{eq:parabola}) heading toward the parabola's vertex $(\chi,\psi)=(1,0)$, which is the collision (bifurcation) point. After bifurcation, we already know (e.g. from Conclusion (vi) above) that the two move toward opposite directions along the real-$\chi$ axis. For $n=19$, the situation is illustrated in Fig. \ref{fig:Fig9}, which shows the parabola as a dashed line, as well as the actual trajectories as the solid line. As usual, the solid line was found by numerically computing the eigenvalues. The arrows give the direction of movement and correspond to increasing $d$ and increasing $y$. Let us add that, before bifurcation, the more general situation described by (\ref{eq:reallambda}) and (\ref{eq:imaglambda}) also yields a trajectory that is parabolic. In the more general case, after bifurcation the eigenvalues are, typically, not \textit{exactly} real.

\section{Purely imaginary $\rho_c$: Large-$n$ formulas}
\label{section:large-n}

Figs. \ref{fig:Fig7}--\ref{fig:Fig9} demonstrate that the formulas in Section \ref{section:purely-imaginary-rho} can give satisfactory numerical results even for large values of $n$, provided that we remain close to the point $y_n$ of interest. We thus develop large-$n$ formulas for the Puiseux-series parameters $y_{n}$, $a_{n}$, and $b_{n}$.

 We first postulate a large-$n$ solution of (\ref{eq:v-parameter}) in the form $v_n\cong \alpha\ln(\beta n)/n$. Since the left-hand side of (\ref{eq:v-parameter}) behaves like $(\beta n)^\alpha/2$ and the RHS like $n$, equating the two sides gives $\alpha=1$ and $\beta=2$. We have thus obtained
\begin{equation}
\label{eq:vn-large-n}
v_n\cong \frac{\ln(2n)}{n}.
\end{equation}
Eqn. (\ref{eq:vn-large-n}) shows that the arguments of the hyperbolic functions in (\ref{eq:yn-in-terms-of-vn}) are large. We thus replace those functions by their large-argument
approximations to get $y_n\cong\exp{(\frac{n+1}{2}v_n)}/\exp{(\frac{n-1}{2}v_n)}=\exp(v_n)$. We then invoke $(\ref{eq:vn-large-n})$ once more to get 
\begin{equation}
\label{eq:y-large-n}
y_n\cong (2n)^{1/n}.
\end{equation}

The small-$v_n$ approximations to the RHSs of (\ref{eq:xn-in-terms-of-vn}) and (\ref{eq:sqrt-in-terms-of-vn}) yield 
\begin{equation}
\label{eq:xn-sqrt-large-n}
x_n\cong 1+\frac{1}{2}v_n^2\cong 1+\frac{1}{2}\left[\frac{\ln(2n)}{n}\right]^2,\quad \sqrt{x_n^2-1}\cong v_n\cong \frac{\ln(2n)}{n}.
\end{equation}
Since $(n-1)v_n$ and $nv_n$ are large, (\ref{eq:tt-in-terms-of-vn}), (\ref{eq:uu-in-terms-of-vn}), and (\ref{eq:vn-large-n}) give
\begin{equation}
\label{eq:tt-uu-large-n}
T_{n-1}(x_n)=\frac{1}{2}\exp[(n-1)v_n]\cong n-\ln(2n),\quad
U_{n-1}(x_n)=\frac{\exp(nv_n)}{2 v_n}\cong\frac{n^2}{\ln(2n)}.
\end{equation} 

The last step is to substitute (\ref{eq:xn-sqrt-large-n})--(\ref{eq:tt-uu-large-n}) into (\ref{eq:mag-gamma-for-imag}) and (\ref{eq:mag-delta-for-imag}). The final results thus obtained are
\begin{equation}
\label{eq:a-large-n}
a_{n}\cong\sqrt{2n}-
\frac{\ln(2n)+1}{\sqrt{2n}}
\end{equation}
and
\begin{equation}
\label{eq:b-large-n}
b_{n}\cong\frac{4}{3}n-\frac{4}{3}[\ln(2n)+1],
\end{equation}
where we discarded terms of high order.

Numerically, the errors in (\ref{eq:y-large-n}), (\ref{eq:a-large-n}), and (\ref{eq:b-large-n})  are, respectively,  5.2\%, 0.4\%, and 1.7\%, when $n=19$. The errors decrease to 2\%, 0.04\%, and 0.3\% when $n=55$, and to 0.6\%, 0.002\%, and 0.06\% when $n=155$.

\section{Discussion}

The small-$|\epsilon|$ studies of this paper resulted in low-order formulas, with errors of $O\left(|\epsilon|^{3/2}\right)$. Furthermore, with the exception of Section \ref{section:large-n}, we assumed that $n$ is fixed. It would thus be illuminating to carry out, if possible, similar studies which would assume that $n$ is large from the outset. A natural first goal would be to find simple, large-$n$  approximations to the global borderline curves $B_n^{(1)}$ and $B_n^{(2)}$. 

Large-$n$ analyses pertaining to $n\times n$ Toeplitz matrices usually start from the  symbol of the associated Laurent matrix \cite{Bottcher1, Bottcher2}. As we already mentioned in the Introduction, the Laurent matrix corresponding to $K_n(\rho)$ has a well-defined and bounded symbol only when $|\rho|<1$. Numerical results indicate that  all $B_n^{(k)}$ lie \textit{outside} the unit circle; see, e.g., Figs. \ref{fig:Fig1}--\ref{fig:Fig3} and Fig. \ref{fig:Fig8}. It thus seems that the aforesaid symbol cannot help in approximating $B_n^{(k)}$. Instead, the obvious starting point is the 
\textit{exact} formula for $B_n^{(k)}$, see Corollary 4.3 of \cite{FikMavr}, which is the formula that produced the $B_n^{(k)}$ in the aforementioned figures.

\section{Acknowledgments}

Thanks to Ioannis Chremmos, Dionisios Margetis, and Christos Papapanos for useful discussions.

\appendix

\section{Useful properties of Chebyshev polynomials}
\label{app:b}

This appendix lists certain properties of Chebyshev polynomials. All can be found in \cite{Wolfram}, or follow easily from the definitions (\ref{eq:chebyshev}). Throughout, we assume $z\in\mathbb{C}$. The square root $\sqrt{1-z^2}$ is single-valued, but its choice is arbitrary.
\begin{equation}
\label{eq:b7}
T_{2k}(z)=2T_k^2(z)-1,\quad k=0,1,2,\ldots;
\end{equation}
\begin{equation}
\label{eq:b8}
T_{k}(z)=2zT_{k+1}(z)-T_{k+2}(z),\quad k=0,1,2,\ldots;
\end{equation}
\begin{equation}
\label{eq:b2}
T_{k}(x)=2^{k-1}x^k+o(x^{k}),\quad x\to+\infty,\quad k=1,2,\ldots;
\end{equation}
\begin{equation}
\label{eq:b5}
U_k(1)=k+1,\quad k=0,1,2,\ldots;
\end{equation}
\begin{equation}
\label{eq:b4}
T_k(1)=1,\quad k=0,1,2,\ldots;
\end{equation}
\begin{equation}
\label{eq:b6}
U_k(x)>k+1,\quad T_k(x)>1,\quad (x>1,\quad k=1,2,\ldots);
\end{equation}
\begin{equation}
\label{eq:b3}
\frac{d}{dz}T_k(z)=kU_{k-1}(z),\quad k=0,1,2,\ldots;
\end{equation}
\begin{equation}
\label{eq:b101}
T_{k}(\cosh x)=\cosh(kx),\quad x>1,\quad k=0,1,2,\ldots;
\end{equation}
\begin{equation}
\label{eq:b51}
U_{2k}(-z)=U_{2k}(z),\quad k=0,1,2,\ldots;
\end{equation}
\begin{equation}
\label{eq:b11}
U_{k}\left(\sqrt{1-z^2}\right)=
\begin{cases}
(-1)^\frac{k}{2}\,z^{-1}\,T_{k+1}(z),\quad k=0,2,4,\ldots,\\
(-1)^\frac{k-1}{2}\,\frac{\sqrt{1-z^2}}{z}\,U_{k}(z),\quad k=1,3,5,\ldots;
\end{cases}
\end{equation}
\begin{equation}
\label{eq:b1}
T_{k}\left(\sqrt{1-z^2}\right)=
\begin{cases}
(-1)^\frac{k}{2}\,T_{k}(z),\quad k=0,2,4,\ldots,\\
(-1)^\frac{k-1}{2}\,\sqrt{1-z^2}\,U_{k-1}(z),\quad k=1,3,5,\ldots;
\end{cases}
\end{equation}
\begin{equation}
\label{eq:b9}
(z^2-1)\left[U_{k-1}(z)\right]^2=\left[T_k(z)\right]^2-1,\quad k=1,2,3,\ldots;
\end{equation}
\begin{equation}
\label{eq:b102}
U_{k}(\cosh x)=\frac{\sinh[(k+1)x]}{\sinh x},\quad x>1,\quad k=0,1,2,\ldots.
\end{equation}

\bibliographystyle{amsplain}

\end{document}